\documentclass[10pt]{article}
\usepackage{amsmath, amsthm, amssymb, amsfonts, mathrsfs, mathtools}
\usepackage{graphicx}
\usepackage{makeidx}
\usepackage[left=2cm,right=2cm,top=2cm,bottom=2cm]{geometry}

\usepackage{caption}
\usepackage{subcaption}

\usepackage{tikz}
\usetikzlibrary{decorations.pathreplacing}
\tikzstyle{vertex}=[auto=left,circle,draw=black,fill=white, inner sep=1.5]
\usepackage{enumitem}
\usepackage{makecell}
\usepackage{blindtext}

\usepackage{hyperref}
\hypersetup{colorlinks=true, citecolor={blue}, linkcolor=blue, filecolor=magenta, urlcolor=cyan}

\newtheorem{theorem}{Theorem}

\newtheorem{remark}{Remark}

\newtheorem{lemma}{Lemma}
\newtheorem{definition}{Definition}
\newtheorem{example}{Example}[section] 
\newtheorem{problem}{Problem}
\newtheorem{conjecture}{Conjecture}

\title{RNA Number of Some Parity Signed Generalized Petersen Graphs}

\author{ Deepak Sehrawat\\
Department of Mathematics\\
Indian Institute of Technology Guwahati\\
Guwahati, India - 781039\\
Email: deepakmath55555@iitg.ac.in\\
\\
Bikash Bhattacharjya\\
Department of Mathematics\\
Indian Institute of Technology Guwahati\\
Guwahati, India - 781039\\
Email: b.bikash@iitg.ac.in
}
\date{}

\begin{document}
\maketitle

\vspace{-0.3in}

\vspace*{0.3in}
\noindent
\textbf{Abstract.} A signed graph $\Sigma=(G,\sigma)$ is said to be \textit{parity signed} if there exists a bijection $f : V(G) \rightarrow \{1,2,...,|V(G)|\}$ such that $\sigma(uv)=+$ if and only if $f(u)$ and $f(v)$ are of same parity, where $uv$ is an edge of $G$. The \textbf{rna} number of a graph $G$, denoted $\sigma^{-}(G)$, is the minimum number of negative edges among all possible parity signed graphs over $G$. The \textbf{rna} number is also equal to the minimum cut size that has nearly equal sides.

In this paper, for generalized Petersen graph $P(n,k)$, we prove that $3 \leq \sigma^{-}(P(n,k)) \leq n$ and these bounds are sharp. The exact value of $\sigma^{-}(P(n,k))$ is determined for $k=1,2$. Some famous generalized Petersen graphs namely, Petersen graph $P(5,2)$, D\" urer graph $P(6,2)$, M\" obius-Kantor graph $P(8,3)$, Dodecahedron $P(10,2)$, Desargues graph $P(10,3)$ and Nauru graph $P(12,5)$ are also treated.

We show that the minimum order of a $(4n-1)$-regular graph having \textbf{rna} number one is bounded above by $12n-2$. The sharpness of this upper bound is also shown for $n=1$. We also show that the minimum order of a $(4n+1)$-regular graph having \textbf{rna} number one is $8n+6$. Finally, for any simple connected graph of order $n$, we propose an $O(2^n + n^{\lfloor \frac{n}{2} \rfloor})$ time algorithm for computing its \textbf{rna} number.

\vspace*{0.05in}

\vspace*{0.1in}

\section{Introduction}\label{intro}
All graphs and signed graphs considered in this paper are simple, connected and undirected. For all the graph theoretic terms that are used here, we refer the reader to~\cite{Bondy}. 

A signed graph is a graph whose edges are either positive or negative. Harary~\cite{Harary} introduced the concept of signed graphs and since then signed graphs have been considered to be a natural generalization of ordinary graphs. 

Recently, Acharya and Kureethara~\cite{Acharya1} introduced a special type of signed graph called the parity signed graph. In~\cite{Acharya2},  Acharya, Kureethara, and Zaslavsky have shown that parity signed graphs also have sociological aspects. This concept is based on the assignment of integers $\{1,2,...,|V(G)|\}$ to the vertices of a graph $G$. It is equivalent to a partition of the vertex set of a graph into two subsets, $A$ and $B$, such that $||A|-|B|| \leq 1$. In~\cite{Acharya2}, ~the authors characterized some families of parity signed graphs, namely, signed stars, bistars, cycles, paths and complete bipartite graphs. 

The \textbf{rna} number of a graph $G$, denoted $\sigma^{-}(G)$, is the minimum number of negative edges among all the possible parity signed graphs over $G$. It is equal to the size of a minimum cut whose sides are nearly equal. The \textbf{rna} number of some families of graphs such as stars, wheels, paths, cycles and complete graphs have been examined. For details, see~\cite{Acharya1, Acharya2}. 

This paper is organized as follows. In Section~\ref{prelim}, we give necessary definitions and existing results. In Section~\ref{section of GPG and forbidden cuts}, generalized Petersen graphs and their cuts with equal sides are discussed. In Section~\ref{main result}, we consider \textbf{rna} number of generalized Petersen graphs and we show that

\begin{enumerate}
\item[(1)] the \textbf{rna} number for the class of all generalized Petersen graphs $P(n,k)$ lies between 3 and $n$. Sharpness of the upper bound is obtained by showing that the Petersen graph $P(5,2)$ has \textbf{rna} number 5, 

\item[(2)] $\sigma^{-}(P(n,1))=5$ for odd $n \geq 5$, and $\sigma^{-}(P(n,1))=4$ for even $n \geq 4$; and

\item[(3)] $\sigma^{-}(P(n,2))=7$ for odd $n \geq 7$, and $\sigma^{-}(P(n,2))=6$ for even $n \geq 8$.
\end{enumerate}

In Section~\ref{section of rna no of famous GPGs}, we show that the \textbf{rna} number of Petersen graph, D\" urer graph, M\" obius-Kantor graph, Dodecahedron, Desargues graph and Nauru graph are 5, 4, 6, 6, 8 and 8, respectively. In Section~\ref{section of smallest cubic graphs}, we show that the minimum order of a $(4n-1)$-regular graph having \textbf{rna} number one is bounded above by $12n-2$. Also, a unique cubic graph is constructed having \textbf{rna} number one that reach this bound. We also show that the minimum order of a $(4n+1)$-regular graph having \textbf{rna} number one is $8n+6$. In Section~\ref{section of time complexity}, we propose an exponential time algorithm for computing the \textbf{rna} number $\sigma^{-}(G)$ of $G$. In Section~\ref{section of conclusion}, we conclude the paper and propose a conjecture which states that the \textbf{rna} number of a graph can be computed in polynomial time.

\section{Preliminaries}\label{prelim}
A graph $G = (V(G), E(G))$ is an ordered pair, where $V(G)$ and $E(G)$ represent the vertex set and the edge set of $G$, respectively. By $|V(G)|$ and $|E(G)|$, we denote the order and the size of $G$, respectively. An edge joining the vertices $x$ and $y$ is denoted by $xy$. The length of a shortest path joining the vertices $x$ and $y$, denoted $d_{G}(x,y)$, is called the \textit{distance} between $x$ and $y$. The $k$-\textit{th power} of a simple graph $G$ is the graph $G^k$ whose vertex set is $V(G)$, and two distinct vertices are adjacent in $G^k$ if and only if their distance in $G$ is at most $k$.

An edge $e$ of a graph $G$ is said to be a \textit{cut edge} of $G$ if deletion of $e$ results in a disconnected graph. An \textit{edge cut} (or simply a \textit{cut}) of $G$ is a set of edges whose deletion results in a disconnected graph. In $G$, a cut whose edges lie between vertices of $A$ and $A^c$ for some $A \subset V(G)$ is denoted by $[A : A^c]$. The \textit{size} of the cut $[A : A^c]$ is the number of edges in $[A : A^c]$ and is denoted by $|[A : A^c]|$. A cut of odd size is said to be an \textit{odd cut} and of even size is said to be an \textit{even cut}. The numbers $|A|$ and $|A^c|$ are called the sides of the cut $[A : A^c]$. The \textit{edge-connectivity} $\kappa'(G)$ of a graph $G$ is the minimum size of cut. For a connected graph $G$ with minimum degree $\delta$, it is well known that $1 \leq \kappa'(G) \leq \delta$.

A \textit{signed graph} $\Sigma=(G,\sigma)$ consists of a graph $G=(V,E)$ and a sign function $\sigma$ which labels each edge of $G$ by $+~\text{or}~-$ sign. The graph $G$ is called the \textit{underlying graph} of $\Sigma$. An edge is called \textit{positive} if $\sigma(e)=+$, and \textit{negative} otherwise. The set of negative edges of $\Sigma$ is $E^{-}(\Sigma)$ and the set of positive edges is $E^{+}(\Sigma)$. A signed graph $\Sigma$ is said to be \textit{all-positive} if $E^{-}(\Sigma) = \emptyset$ and \textit{all-negative} if $E^{+}(\Sigma) = \emptyset$. A signed graph is \textit{homogeneous} if it is either all-positive or all-negative, and \textit{heterogeneous} otherwise.

Now we give some necessary definitions and results.

\begin{definition}\cite{Acharya2}
\rm{Given a graph $G$ of order $n$ and a bijection $f : V(G) \rightarrow \{1,2,...,n\}$, define $\sigma_{f} : E(G) \rightarrow \{+,-\}$ such that $\sigma_{f}(uv)=+$ if $f(u)$ and $f(v)$ are of same parity and $\sigma_{f}(uv)=-$ if $f(u)$ and $f(v)$ are of different parity, where $uv$ is an edge in $G$. We define $\Sigma_{f}$ to be the signed graph $(G,\sigma_{f})$.} 
\end{definition}

\begin{definition}\cite{Acharya1}
\rm{A signed graph $\Sigma = (G,\sigma)$ is called a \textit{parity signed graph}, if there exists a bijection $f : V(G) \rightarrow \{1,2,...,n\}$ such that $ \sigma = \sigma_{f}$.} 
\end{definition}

A signed graph is said to be \textit{balanced} if every cycle in it has an even number of negative edges. Harary introduced this idea in~\cite{Harary}. In~\cite{Acharya1}, the authors proved that a parity signed cycle is always balanced. Consequently, every parity signed graph is balanced, see~\cite[Theorem 2.1]{Acharya2}.

\begin{definition}\cite{Acharya1}
\rm{The \textbf{rna} number of a graph $G$, denoted $\sigma^{-}(G)$, is the minimum number of negative edges among all possible parity signed graphs over $G$.} 
\end{definition}

Note that finding the minimum number of negative edges among all parity signed graphs over a graph $G$ is equivalent to finding the size of a minimum cut of $G$ with nearly equal sides. More precisely, if $G$ is of even order then it is equivalent to find the size of a minimum cut whose sides are equal and if $G$ is of odd order then it is equivalent to find the size of a minimum cut whose sides differ by exactly one. 

We now mention some known results about the \textbf{rna} number of some graphs, \emph{viz}., paths, cycles, stars, wheels and complete graphs.

\begin{theorem}\cite{Acharya1}
For any path $P_n$ of order $n$, $\sigma^{-}(P_n) = 1$.
\end{theorem}\label{rna of path}

\begin{theorem}\cite{Acharya1}
For any cycle $C_n$ with $n$ vertices, $\sigma^{-}(C_n) = 2$.
\end{theorem}\label{rna of cycle}

\begin{theorem}\cite{Acharya1}
For a star $K_{1,n}$ with $n+1$ vertices, $\sigma^{-}(K_{1,n}) = \lceil \frac{n}{2} \rceil $.
\end{theorem}\label{rna of star}

A \textit{wheel} $W_n$ is the edge-disjoint union of $C_{n-1}$ and $K_{1,n-1}$. The \textbf{rna} number of a wheel $W_n$ is determined in \cite[Theorem 12]{Acharya2}.

\begin{theorem}\cite{Acharya2}
For a wheel $W_n$, $\sigma^{-}(W_{n}) = \lceil \frac{n+4}{2} \rceil $.
\end{theorem}\label{rna of wheel}

\begin{theorem}\cite{Acharya1}
For a complete graph $K_n$ with $n \geq 2$ vertices, $\sigma^{-}(K_{n}) = \lceil \frac{n}{2} \rceil \lfloor \frac{n}{2} \rfloor $.
\end{theorem}\label{rna of complete}

\section{Generalized Petersen Graphs and Their Forbidden Cuts}\label{section of GPG and forbidden cuts}

The family of generalized Petersen graphs was introduced by Coxeter~\cite{Coxeter} in 1950 and was given its name by Watkins~\cite{Watkins} in 1969.

\begin{definition}
\rm{For any $n \geq 3$ and $k \geq 1$ with $2k < n$, the \textit{generalized Petersen graph} $P(n,k)$ has vertex set $V(P(n,k))=\{u_{i}, v_{i}~|~i =0,1,...,n-1\}$ and edge set $E(P(n,k))=\{u_{i}u_{i+1}, u_{i}v_{i}, v_{i}v_{i+k}~|~i = 0,1,...,n-1\}$, where subscripts are read modulo $n$.}
\end{definition}

From the definition, it is clear that $P(n,k)$ is a cubic graph and $P(5,2)$ is the well-known Petersen graph. The vertices $u_{0}, u_{1},...,u_{n-1}$ are called $u\text{-}vertices$ and the vertices $v_{0},v_{1},...,v_{n-1}$ are called $v\text{-}vertices$. The edges $u_{i}v_{i}$, for $i \in \{0,1,...,n-1\}$, are called \textit{spokes} and the set of all spokes is denoted by $S_{s}$. We call the vertex $u_i~(v_{i})$ is the \textit{partner} of the vertex $v_i~(u_{i})$, for each $i \in \{0,1,...,n-1\}$.

The cycle induced by all $u$-vertices is called the \textit{outer cycle} of $P(n,k)$ and is denoted by $C_{o}$. The cycle(s) induced by all $v$-vertices is(are) called the \textit{inner cycle(s)} of $P(n,k)$. If $\text{gcd}(n,k) = d$ then the subgraph induced by all $v$-vertices consists of $d$ pairwise disjoint $\frac{n}{d}$-cycles. If $d >1$ then no two vertices among $v_{0}, v_{1},..., v_{d-1}$ can be in the same $\frac{n}{d}$-cycle. For $d =1$, $P(n,k)$ has only one inner cycle, and in this case the inner cycle is denoted by $C_I$. If $d > 1$, then $P(n,k)$ has $d$ inner cycles, and these inner cycles are denoted by $C_1,C_2,...,C_d$ such that $v_i \in V(C_{i+1})$, for $i \in  \{0,1,2,...,d-1\}$. 

\begin{definition}\label{def of induced subgraph}
\rm{For any subset $A$ of vertices of $G$, the \textit{induced subgraph} $G[A]$ is the subgraph of $G$ whose vertex set is $A$ and whose edge set consists of all edges of $G$ having both end points in $A$.}
\end{definition}

\begin{lemma}\label{killing possibility 3}
For any $n\geq 4$ and $k \geq 1$, $P(n,k)$ cannot have a cut of size three with equal sides.
\end{lemma}
\begin{proof}
We analyse two cases depending upon whether $n$ is odd or even.

\noindent
\textbf{Case 1.} Let $n=2l$, for some $l \geq 2$. Let there exist a subset $A \subset V(P(2l,k))$ such that $|A| =  2l$ and $|[A : A^{c}]| = 3$. Denote the degree of a vertex $a$ in $P(2l,k)[A]$ by $d_{A}(a)$. We have $$\sum_{a \in A} d_{A}(a) = 3(2l) - 3,~ \text{an odd number}.$$  

This shows that $P(2l,k)[A]$ does not satisfy the handshaking lemma. Hence no such $A$ is possible.\\
\noindent
\textbf{Case 2.} Let $n=2l+1$, for some $l \geq 2$. Let there exist a subset $A \subset V(P(2l+1,k))$ such that $|A| =  2l+1$ and $|[A : A^{c}]| = 3$. If $A$ contains either all $u$-vertices or all $v$-vertices then all the spokes will be in $[A : A^{c}]$. This contradicts the fact that $|[A : A^{c}]| = 3$. Therefore $A$ must contain $u$-vertices as well as $v$-vertices. Consequently $[A : A^{c}]$ contains at least two edges of $C_o$, since $u$-vertices lie in both $A$ and $A^c$.

Now we consider two sub-cases.

\textit{Subcase 2(i).} Let $\gcd(2l+1,k)=1$, so that $P(2l+1,k)$ has exactly one inner cycle $C_I$. The condition that $v$-vertices lie in both $A$ and $A^c$ will insist $[A : A^{c}]$ to contain at least two edges of $C_I$. Thus we have $|[A : A^{c}]| \geq 4$, a contradiction to the fact that $|[A : A^{c}]| = 3$. 

\textit{Subcase 2(ii).} Let $\gcd(2l+1,k)=d \geq 2$ so that $P(2l+1,k)$ has $d$ inner cycles $C_{1},C_{2},...,C_{d}$ each of length $\frac{2l+1}{d}$. Note that each $C_{i}$ lies completely either in $A$ or in $A^c$, otherwise we will get a contradiction on the size of $|[A : A^{c}]|$. Therefore at least one cycle among $C_{1},C_{2},...,C_{d}$ lies in $A$ and at least one lies in $A^c$. Hence both $A$ and $A^c$ contain at least three $u$-vertices. 

It is easy to see that if $|[A : A^{c}]| = 3$ then exactly two edges of $[A : A^{c}]$ must be edges of $C_o$ and the third edge must be a spoke. Let this spoke be $u_{j}v_{j}$, for some $j \in \{0,1,2,...,2l\}$. Without loss of generality, let $u_{j} \in A$ and $v_{j} \in A^c$. As exactly one spoke lies across $A$ and $A^c$, the remaining $u$-vertices of $A$ must have their partners in $A$. Thus the number of $u$-vertices and $v$-vertices in $A$ are $l+1$  and $l$, respectively. As $[A : A^{c}]$ has exactly two edges of $C_o$, there exists a path of length $l$ induced by the $u$-vertices of $A$. Let the end points of this path be $u_{r}$ and $u_{r+l}$, for some $r \in \{0,1,2,...,2l\}$. Consequently, the set of $v$-vertices in $A$ is $\{v_{r},v_{r+1},...,v_{r+l} \} \setminus \{v_j\}$ and the subgraph induced by these $v$-vertices must be edge-disjoint union of some inner cycle(s) of length $\frac{2l+1}{d}$. 

The condition $2k < 2l+1$, together with $\gcd(2l+1,k) = d \geq 2$, implies that $3 \leq k \leq l$. Also there must be an inner cycle $v_{r}v_{r+k}v_{r+2k}...v_{r+(2l+1)-2k}v_{r+(2l+1)-k}v_{r}$ containing the vertex $v_{r}$. Now if $r \leq k$ then $r+(2l+1)-k \geq r+ (l+1)$, since $k \leq l$. If $r > k$ then $r+(2l+1)-k = r-k<r$. Thus in both cases, the vertex $v_{r+(2l+1)-k}$ does not lie in $A$. Therefore at least two edges of the inner cycle containing $v_{r}$ must lie in $[A:A^c]$. This gives $|[A : A^c]| \geq 5$, a contradiction. Note that if $r=j$, then we can consider the inner cycle containing $v_{r+1}$ and get a similar contradiction.

This completes this proof.
\end{proof}

\begin{lemma}\label{killing possibility of odd cuts}
For any even $n \geq 4$, $P(n,k)$ cannot have an odd cut with equal sides.
\end{lemma}
\begin{proof}
Take $n = 2l$ for some $l \geq 2$. On the contrary, let $P(n,k)$ have an odd cut with equal sides. Thus there exists a subset $A \subset V(P(2l,k))$ such that $|A| = 2l$ and $|[A : A^c]| = 2r+1$, for some $r \geq 1$.

If $d_{A}(a)$ is the degree of vertex $a$ in $P(2l,k)[A]$, then we have $$ \sum_{a \in A} d_{A}(a) = 3(2l) - (2r+1), ~\text{an odd number}.$$
This is a contradiction to the handshaking lemma. Hence no odd cut with equal sides is possible in $P(2l,k)$. This completes the proof.
\end{proof}

\begin{lemma}\label{killing possibility of even cuts}
For any odd $n \geq 5$, $P(n,k)$ cannot have an even cut with equal sides.
\end{lemma}
\begin{proof}
Take $n = 2l+1$ for some $l \geq 2$. Let, if possible, $P(2l+1,k)$ have an even cut with equal sides. Thus there exists a subset $A \subset V(P(2l+1,k))$ such that $|A| = 2l+1$ and $|[A : A^c]| = 2r$, for some $r \geq 2$. Here $r$ cannot be one, since no edge cut of size less than three is possible due to the edge-connectivity of generalized Petersen graphs.

If $d_{A}(a)$ is the degree of vertex $a$ in $P(2l+1,k)[A]$, then we have $$ \sum_{a \in A} d_{A}(a) = 3(2l+1) - (2r) = \text{an odd number},$$
a contradiction. Hence no even cut with equal sides is possible in $P(2l+1,k)$. This completes the proof.
\end{proof}

\section{Main Results}\label{main result}
A simple but important result is the following. 

\begin{theorem}\label{gen lower bound}
Let $G$ be a graph with edge-connectivity $k$. Then $\sigma^{-}(G) \geq k$.
\end{theorem}
\begin{proof}
Clearly, no cut of $G$ with nearly equal sides can have less than $k$ edges since $\kappa'(G) = k$. Hence $\sigma^{-}(G) \geq k$. 
\end{proof} 

\begin{theorem}\label{sharp bound of P(n,k)}
For any $n \geq 3$ and $k \geq 1$, the \textbf{rna} number of $P(n,k)$ satisfies $$3 \leq \sigma^{-}(P(n,k)) \leq n.$$
\end{theorem}
\begin{proof}
The lower bound directly follows from Theorem~\ref{gen lower bound} as the edge-connectivity of the generalized Petersen graph is three. 

Define $f : V(P(n,k)) \rightarrow \{1,2,...,2n\}$ such that $f(u_{i}) = 2i+1$ and $f(v_{i}) = 2i+2$, for $0 \leq i \leq n-1$. This labeling $f$ induces the parity signed graph $(P(n,k), \sigma_f)$. It is clear that the number of negative edges in $(P(n,k), \sigma_f)$ is $n$. Hence $ \sigma^{-}(P(n,k)) \leq n$.
\end{proof}

It is proved (in Lemma~\ref{rna of P(3,1)} and Example~\ref{the rna number of P(5,2)}, respectively) that $ \sigma^{-}(P(3,1)) = 3$ and $ \sigma^{-}(P(5,2)) = 5$. Thus bounds of Theorem~\ref{sharp bound of P(n,k)} are sharp.

\begin{theorem}\label{rna number of P(n,k) with gcd(n,k) = 1}
Let $n \geq 5$ and $k \geq 2$ be such that $\gcd(n,k)=1$. Then the \textbf{rna} number of $P(n,k)$ satisfies $$5 \leq \sigma^{-}(P(n,k)) \leq n.$$
 \end{theorem}
\begin{proof} The upper bound is given by Theorem~\ref{sharp bound of P(n,k)}. 

For odd $n$, the lower bound simply follows from Lemma~\ref{killing possibility 3} and Lemma~\ref{killing possibility of even cuts}. Let $n$ be an even number. By Lemma~\ref{killing possibility 3}, $\sigma^{-}(P(n,k)) \geq 4$. Now we show that the \textbf{rna} number of $P(n,k)$ can not be four.

By contradiction, let $P(n,k)$ has a cut of size four with equal sides. Then there exists a subset $A \subset V(P(n,k))$ such that $|A|=n$ and $|[A : A^c]|=4$. If $A$ contains either all $u$-vertices or all $v$-vertices then $[A : A^c]$ has precisely $n$ spokes of $P(n,k)$. Hence $|[A : A^c]|=n$, a contradiction to the fact that $|[A : A^c]|=4$ and $n \geq 5$. Therefore, $A$ must contain $u$-vertices as well as $v$-vertices.

Since $\gcd(n,k) = 1$, there is only one inner cycle $C_{I}$ in $P(n,k)$. Note that $[A : A^c]$ has to contain exactly two edges of $C_o$ and two edges of $C_I$ because both $A$ and $A^c$ contain vertices of $C_o$ and $C_I$. Also $A$ contains as many $u$-vertices as $v$-vertices. Otherwise, $[A : A^c]$ will contain at least one spoke and we will get a contradiction on $|[A : A^c]|$. Further, the condition ``$[A : A^c]$ contains exactly two edges of $C_I$'' insists $v$-vertices of $A$ to induce a path of order $\frac{n}{2}$. Let this path be $P: v_{r}v_{r+k}...v_{r+(\frac{n}{2}-1)k}$, for some $r \in \{0,1,2,...,n-1\}$. 

Similarly, all $u$-vertices of $A$ induce a path of order $\frac{n}{2}$ and let this path be $P': u_{j}u_{j+1}...u_{j+(\frac{n}{2}-1)}$, for some $j \in \{0,1,2,...,n-1\}$. As $k \geq 2$, all these $u$-vertices and $v$-vertices of the paths $P$ and $P'$ cannot be partners of each other, for any $r,j \in \{0,1,2,...,n-1\}$. Hence at least two spokes lie across $A$ and $A^c$. Therefore, $|[A : A^c]| \geq 6$, a contradiction. This completes the proof.
 \end{proof} 
 
\begin{remark}\label{remark 1} \rm{From Theorem~\ref{sharp bound of P(n,k)} and Lemma~\ref{killing possibility 3}, it is easy to see that $\sigma^{-}(P(n,k)) \geq 4$, for $n \geq 4$.}

\end{remark} 
 
\begin{theorem}\label{rna number of P(2l,k) with gcd(n,k) = 1}
Let $n \geq 8$ be even and $k \geq 3$ be such that $\gcd(n,k)=1$. Then the \textbf{rna} number of $P(n,k)$ satisfies $$6 \leq \sigma^{-}(P(n,k)) \leq n.$$
\end{theorem}
\begin{proof}
The result follows from Theorem~\ref{rna number of P(n,k) with gcd(n,k) = 1} and Lemma~\ref{killing possibility of odd cuts}.
\end{proof}
 
\subsection{The \textbf{rna} Number of $P(n,1)$}

\begin{lemma}\label{rna of P(3,1)}
The \textbf{rna} number of $P(3,1)$ is three. 
\end{lemma}
\begin{proof}
Since the edge-connectivity of $P(3,1)$ is three, Theorem~\ref{gen lower bound} gives $\sigma^{-}(P(3,1)) \geq 3$. We now label the vertices of $P(3,1)$ through a bijection $f : V(P(3,1)) \rightarrow \{1,2,...,6\}$ such that the number of negative edges in the induced $(P(3,1), \sigma_{f})$ is exactly three.

Define $f : V(P(3,1)) \rightarrow \{1,2,...,6\}$ such that $f(u_i) = 2i+1$ and $f(v_i) = 2i+2$, for $0 \leq i \leq 2$. Let $A=\{u_{0},u_{1},u_{2}\}$ and $B=\{v_{0},v_{1},v_{2}\}$. Clearly all the spokes of $P(3,1)$ are negative and all the edges of the $C_o$ and $C_I$ are positive in $(P(3,1), \sigma_{f})$. Thus $\sigma^{-}(P(3,1)) = 3$.
\end{proof}

In the subsequent discussion, we consider $n \geq 4$.

\begin{theorem}\label{rna number of P(n,1)}
For $n \geq 4$, we have 
$$\sigma^{-}(P(n,1)) = 
\begin{cases} 
4~~ \text{if $n$ is even},\\
5~~ \text{if $n$ is odd}.\\
\end{cases}$$
\end{theorem}
\begin{proof} We analyse two cases depending upon whether $n$ is even or odd.\\
\noindent
\textit{Case 1.} Let $n=2l$, for some $l \geq 2$. By Remark~\ref{remark 1}, we have $\sigma^{-}(P(2l,1)) \geq 4$. Thus to show that $\sigma^{-}(P(2l,1)) =  4$, we produce an induced parity signed $P(2l,1)$ which contains exactly four negative edges. Let $f : V(P(2l,1)) \rightarrow \{1,2,...,4l\}$ be defined by
$$ f(u_{i}) = 
 \begin{cases}
2i+2 \hspace{0.44in} ~\text{for}~ 0 \leq i \leq l-1,\\
2i-2l+1 \hspace{0.16in} ~\text{for}~ l \leq i \leq 2l-1;
\end{cases}
$$
and 
$$ f(v_{i}) = 
\begin{cases}
2i+2l+2 \hspace{0.16in} ~\text{for}~ 0 \leq i \leq l-1,\\
2i+1 \hspace{0.44in} ~\text{for}~ l \leq i \leq 2l-1.
\end{cases}
$$ 

 Let $A = \{u_{l},u_{l+1},...,u_{2l-1},v_{l},...,v_{2l-1}\}$ and $B=\{u_{0},u_{1},...,u_{l-1},v_{0},...,v_{l-1}\}$ so that $|A| = |B| = 2l$. Hence, every edge of $P(2l,1)$ having both end points either in $A$ or in $B$ is positive in the induced parity signed labeling. Consequently, all edges of $P(2l,1)$ lying between $A$ and $B$ get negative sign in the induced parity signed labeling. Clearly, $[A : B] = \{u_{2l-1}u_{0},u_{l-1}u_{l}, v_{2l-1}v_{0},v_{l-1}v_{l}\}$. That is, there are exactly four edges between vertices of $A$ and $B$. Hence $\sigma^{-}(P(2l,1)) =  4$.
 
\noindent
\textit{Case 2.} Let $n=2l+1$, for some $l \geq 2$. By Remark~\ref{remark 1} and Lemma~\ref{killing possibility of even cuts}, we have $\sigma^{-}(P(2l+1,1)) \geq 5$. Now we produce an induced parity signed graph over $P(2l+1,1)$ which has exactly five negative edges.  

Let $f : V(P(2l+1,1)) \rightarrow \{1,2,...,4l+2\}$ be defined by
 $$\hspace{-0.19in} f(u_{i}) = 
 \begin{cases}
2i+1 \hspace{0.2in} ~\text{for}~ 0 \leq i \leq l,\\
2i-2l \hspace{0.16in} ~\text{for}~ l+1 \leq i \leq 2l;
\end{cases}
$$
and 
$$\hspace{0.15in} f(v_{i}) = 
\begin{cases}
2i+(2l+3) \hspace{0.2in} ~\text{for}~ 0 \leq i \leq l-1,\\
2i+2 \hspace{0.6in} ~\text{for}~ l \leq i \leq 2l.
\end{cases}
$$

 Let $A = \{u_{0},u_{1},...,u_{l},v_{0},...,v_{l-1}\}$ and $B=\{u_{l+1},u_{1+2},...,u_{2l},v_{l},...,v_{2l}\}$ so that $|A| = |B| = 2l+1$. Note that $[A: B] = \{u_{l}u_{l+1},u_{2l}u_{0},u_{l}v_{l},v_{l-1}v_{l},v_{2l}v_{0}\}$. Each edge of $P(2l+1,1)$, except the five edges of $[A: B]$, is positive in the parity signed graph $(P(2l+1,1), \sigma_f)$. Consequently, the number of negative edges in $(P(2l+1,1), \sigma_f)$ is exactly five. Hence, $\sigma^{-}(P(2l+1,1)) = 5$. This completes the proof.
\end{proof}

\subsection{The \textbf{rna} Number of $P(n,2)$}
In this section, our aim is to prove the following two theorems.
\begin{theorem}\label{rna number of P(2l+1,2)}
For any $l \geq 3$, $\sigma^{-}(P(2l+1,2)) = 7$.
\end{theorem}

\begin{theorem}\label{rna number of P(2l,2)}
For any $l \geq 4$, $\sigma^{-}(P(2l,2)) = 6$.
\end{theorem}

In light of Lemma~\ref{killing possibility of even cuts}, it is clear for $l \geq 3$ that the \textbf{rna} number of $P(2l+1,2)$ cannot be 4 and 6.

\begin{lemma}\label{lemma 2 for rna of P(2l+1,2)}
For any $l \geq 3$, the \textbf{rna} number of $P(2l+1,2)$ cannot be 5.
\end{lemma}
\begin{proof}
We prove that no edge cut of size five with equal sides is possible in $P(2l+1,2)$.

On the contrary, let $P(2l+1,2)$ has a cut of size five with equal sides. Then there exists a subset $A \subset V(P(2l+1,2))$ such that $|A|=2l+1$ and $|[A : A^c]|=5$. If $A$ contains either all $u$-vertices or all $v$-vertices then the set of all spokes of $P(2l+1,2)$ will constitute $[A : A^c]$. This gives $|[A : A^c]| \geq 7$, a contradiction. Therefore, $A$ must contain some $u$-vertices as well as $v$-vertices. 

Clearly, $P(2l+1,2)$ has only one inner cycle $C_I$ induced by $v$-vertices, since $\gcd(2l+1,2) = 1$. Since $A$ contains both $u$-vertices and $v$-vertices, $[A : A^c]$ must consist of two edges of $C_o$, two edges of $C_I$ and one spoke. Let this spoke be $u_{j}v_{j}$, for some $j \in \{0,1,...,2l\}$. Without loss of generality, let $u_{j} \in A$ and $v_{j} \in A^c$. The conditions: (i) $|A|=2l+1$, (ii) $u_{j} \in A$ and (iii) $[A : A^c]$ has exactly one spoke, together imply that the number of $u$-vertices and $v$-vertices in $A$ are $l+1$ and $l$, respectively.

Consequently, there exist two paths $P$ and $P'$ of order $l+1$ and $l$, respectively, induced by $u$-vertices and $v$-vertices of $A$. Let these paths be $P$ and $P'$  be $u_{r}u_{r+1}...u_{r+l-1}u_{r+l}$ and $v_{s}v_{s+2}...v_{s+(2l-2)}$, for some $r,s \in \{0,1,...,2l\}$, and $u_{j}$ be one of the vertices of $P$. 

It is easy to check that all vertices of $P'$ cannot have their partners in the vertices of $P$, for any $r,s \in \{0,1,...,2l\}$. This means at least one $v$-vertex of $A$ has its partner in $A^c$, and consequently two spokes lie in $[A : A^c]$. This gives $|[A : A^c]| \geq 6$, a contradiction to the assumption $|[A : A^c]| = 5$. This establishes the lemma. 
\end{proof}

\noindent
\textbf{Proof of Theorem~\ref{rna number of P(2l+1,2)}.} By Theorem~\ref{rna number of P(n,k) with gcd(n,k) = 1}, Lemma~\ref{killing possibility of even cuts} and Lemma~\ref{lemma 2 for rna of P(2l+1,2)}, we have $\sigma^{-}(P(2l+1,2)) \geq 7$. To complete the proof, we produce a parity signed graph over $P(2l+1,2)$ with exactly seven negative edges. 

Define $f : V(P(2l+1,2)) \rightarrow \{1,2,...,4l+2\}$ by 
 $$\hspace{-0.19in} f(u_{i}) = 
 \begin{cases}
2i+1 \hspace{0.2in} ~\text{for}~ i=0,1,...,l,\\
2i-2l \hspace{0.149in} ~\text{for}~ i=l+1,l+2,...,2l;
\end{cases}
$$
and 
$$\hspace{0.25in} f(v_{i}) = 
\begin{cases}
4l+2 \hspace{0.6in} ~\text{for}~ i=0,\\
2l+(2i+1) \hspace{0.2in} ~\text{for}~ i=1,2,...,l,\\
2i \hspace{0.83in} ~\text{for}~ i=l+1,l+2,...,2l.
\end{cases}
$$

 Let $A$ and $B$ be set of all vertices of $P(2l+1,2)$ labeled with odd and even integers, respectively. Thus $A = \{u_{0},u_{1},...,u_{l},v_{1},...,v_{l}\}$ and $B  = \{u_{l+1},u_{l+2},...,u_{2l},v_{0},v_{l+1},...,v_{2l}\}$ with $|A| = |B| = 2l+1$. Only the edges of $P(2l+1,2)$ between vertices of $A$ and $B$ are negative in $(P(2l+1,2), \sigma_f)$. Note that $[A : B] = \{u_{0}v_{0}, u_{0}u_{2l}, u_{l}u_{l+1}, v_{0}v_{2}, v_{2l}v_{1}, v_{l-1}v_{l+1}, v_{l}v_{l+2}\}$ with $|[A : B]|=7$. Hence the number of negative edges in $(P(2l+1,2), \sigma_f)$ is seven. This proves that $\sigma^{-}(P(2l+1,2)) = 7$.   $\qed$

\begin{lemma}\label{killing possibility of 4 for even n}
For any $l \geq 4$, $P(2l,2)$ cannot have a cut of size four with equal sides.
\end{lemma}
\begin{proof}
On the contrary, let if possible, $P(2l,2)$ have a cut of size four with equal sides. Then there exists a subset $A \subset V(P(2l,2))$ such that $|A|=2l$ and $|[A : A^c]| = 4$. Obviously, $A$ must contain some $u$-vertices as well as $v$-vertices. 

Since $\gcd(2l,2)=2$, $P(2l,2)$ has two inner cycles, $C_1 = v_{0}v_{2}v_{4}...v_{2l-2}v_{0}$ and $C_2 =v_{1}v_{3}v_{5}...v_{2l-1}v_{1}$. Note that $[A : A^c]$ contains at least two edges of $C_o$, since $u$-vertices lie in both $A$ and $A^c$. Thus the followings are the only possibilities for the edges of $[A : A^c]$.
\begin{enumerate}
\item[1.] All four edges of $[A : A^c]$ are edges of $C_o$.
\item[2.] Two edges of $[A : A^c]$ are edges of $C_o$ and remaining two are spokes.
\item[3.] Two edges of $[A : A^c]$ are edges of $C_o$ and remaining two edges are of one of the $C_1$ and $C_2$.
\end{enumerate}

We discuss the above possibilities one by one. 

\noindent
\textit{Case 1.} Let $[A : A^c]$ have four edges of $C_o$. Thus all the vertices of one of the $C_1$ and $C_2$ lie in $A$ and other's vertices lie in $A^c$. Without loss of generality, assume that all the vertices of $C_1$ are contained in $A$. Since no spoke is lying across $A$ and $A^c$, we have $A = \{u_{0},u_{2},u_{4},...,u_{2l-2},v_{0},
v_{2},v_{4},...,v_{2l-2}\}$. Consequently, $A^c = \{u_{1},u_{3},...,u_{2l-1},v_{1},v_{3},
...,v_{2l-1}\}$. Clearly, $[A : A^c]$ contains all the edges of $C_o$. That is, $|[A : A^c]|  = 2l \geq 8$, a contradiction.

\noindent
\textit{Case 2.} Let $[A : A^c]$ have two edges of $C_o$ and two spokes. Since $[A : A^c]$ contains no edge of inner cycles, without loss of generality, assume that $A$ contains all the vertices of $C_1$. That is, $\{v_{0},v_{2},...,v_{2l-2}\} \subseteq A$. 

Consequently, the set of $u$-vertices in $A$ is $ \{u_r\} \cup \{u_{0},u_{2},...,u_{2l-2}\} \setminus \{u_{j}\}$ for some $r \in \{1,3,...,2l-1\}$ and $j \in \{0,2,...,2l-2\}$ because $[A : A^c]$ has exactly two spokes. Equivalently, the set of $u$-vertices in $A$ is $ \{u_r\} \cup \{u_{j+2},u_{j+4},...,u_{j+(2l-2)}\}$ for some $r \in \{1,3,...,2l-1\}$ and $j \in \{0,2,...,2l-2\}$. As $l \geq 4$, it is easy to check that for any $r \in \{1,3,...,2l-1\}$ and for any $j \in \{0,2,...,2l-2\}$, $[A : A^c]$ will contain at least four edges of $C_o$, a contradiction.

\noindent
\textit{Case 3.} Let $[A : A^c]$ have two edges of $C_o$ and two edges of one of the inner cycles. Without loss of generality, assume that two edges of $C_1$ lie in $[A : A^c]$. Thus both $A$ and $A^c$ contain at least one vertex of $C_1$. Clearly, all the vertices of $C_2$ lie in either $A$ or in $A^c$. Hence either $A$ or  $A^c$ contains at least $l+1$ $v$-vertices. If $A$ contains at least $l+1$ $v$-vertices then the number of $u$-vertices in $A$ will be at most $l-1$. This observation shows that at most $l-1$ $u$-vertices of $A$ can have their partners in $A$. Hence at least two spokes will lie across $A$ and $A^c$, contradicting our assumption that $[A : A^c]$ contains no spokes. These contradictions establish the lemma.
\end{proof}

\textbf{Proof of Theorem~\ref{rna number of P(2l,2)}.}
By Lemma~\ref{killing possibility of odd cuts} and Lemma~\ref{killing possibility of 4 for even n}, we have $\sigma^{-}(P(2l,2)) \geq 6$. Let us define $f : V(P(2l,2)) \rightarrow \{1,2,3,...,4l\}$ by 
$$\hspace{0.19in} f(u_{i}) = 
 \begin{cases}
2i+1 \hspace{0.55in} ~\text{for}~ i=0,1,...,l-1,\\
2i-(2l-2) \hspace{0.149in} ~\text{for}~ i=l,l+1,...,2l-1;
\end{cases}
$$
and 
$$\hspace{0.35in} f(v_{i}) = 
\begin{cases}
2l+(2i+1) \hspace{0.2in} ~\text{for}~ i=0,1,...,l-1,\\
2i+2 \hspace{0.6in} ~\text{for}~ i=l,l+1,...,2l-1.
\end{cases}
$$
Let $A = \{u_{0},u_{1},...,u_{l-1},v_{0},v_{1},...,v_{l-1}\}$ and $B  = \{u_{l},u_{l+1},...,u_{2l-1},v_{l},v_{l+1},...,v_{2l-1}\}$ so that $|A| = |B| = 2l$. Consequently, all edges of $P(2l,2)$ between the vertices of $A$ and $B$ are negative in $(P(2l,2), \sigma_f)$. Note that $[A : B] = \{u_{0}u_{2l-1}, u_{l-1}u_{l}, v_{2l-2}v_{0}, v_{2l-1}v_{1}, v_{l-2}v_{l}, v_{l-1}v_{l+1}\}$ with $|[A : B]|=6$. Hence the number of negative edges in $(P(2l,2), \sigma_f)$ is six. This proves that $\sigma^{-}(P(2l,2)) = 6$.
$\qed$

\section{Some Famous Generalized Petersen Graphs}\label{section of rna no of famous GPGs}

In this section, we compute the \textbf{rna} number of some well known generalized Petersen graphs such as Petersen graph, D\" urer graph, M\" obius-Kantor graph, Dodecahedron, Desargues graph and Nauru graph.

\begin{example}\label{the rna number of P(5,2)}
\rm{The generalized Petersen graph $P(5,2)$ is well known as the \textit{Petersen} graph. Since $\gcd(5,2)=1$, Theorem~\ref{rna number of P(n,k) with gcd(n,k) = 1} gives $\sigma^{-}( P(5,2)) \geq 5$. Label the vertices of $P(5,2)$ via $f : V(P(5,2)) \rightarrow \{1,2,3,...,10\}$ defined by $f(u_{i}) = 2i+1$ and $f(v_{i}) = 2i+2$, for $0 \leq i \leq 4$. Thus all the spokes of $P(5,2)$ are negative while other edges of $P(5,2)$ are positive in $(P(5,2), \sigma_{f})$. Hence the \textbf{rna} number of Petersen graph is five. That is, $\sigma^{-}(P(5,2) = 5$. }
\end{example}

\begin{figure}[h]
\begin{center}
\begin{tikzpicture}[scale=0.35]

 \begin{scope}[rotate=90]

		\node[vertex] (v1) at ({360/6*0 }:7cm) {};
		\node[vertex] (v6) at ({360/6*1 }:7cm) {};
		\node[vertex] (v5) at ({360/6 *2 }:7cm) {};
		\node[vertex] (v4) at ({360/6 *3 }:7cm) {};
		\node[vertex] (v3) at ({360/6 *4 }:7cm) {};
		\node[vertex] (v2) at ({360/6 *5}:7cm) {};
        \draw [dashed] ({360/6 *0 +1.5}:7cm) arc  ({360/6*0+1.5}:{360/6*1 -1.5 }:7cm);
        \draw ({360/6 *1 +1.5}:7cm) arc  ({360/6*1+1.5}:{360/6*2 -1.5 }:7cm);
        \draw  ({360/6 *2 +1.5}:7cm) arc  ({360/6*2+1.5}:{360/6*3 -1.5 }:7cm);
        \draw [dashed] ({360/6 *3 +1.5}:7cm) arc  ({360/6*3+1.5}:{360/6*4 -1.5 }:7cm);
        \draw  ({360/6 *4 +1.5}:7cm) arc  ({360/6*4+1.5}:{360/6*5 -1.5 }:7cm);
        \draw  ({360/6*5 +1.5}:7cm) arc  ({360/6*5+1.5}:{360/6*6 -1.5 }:7cm);
        
        \node[yshift=0.32cm] at ({360/6*0 }:7cm) {1};
		\node[xshift=-0.3cm, yshift=0.1cm] at ({360/6*1 }:7cm) {$6$};
		\node[xshift=-0.25cm, yshift=-0.18cm] at ({360/6 *2 }:7cm) {4};
		\node[yshift=-0.3cm] at ({360/6 *3 }:7cm) {2};
		\node[ xshift=0.3cm, yshift = -0.15cm] at ({360/6 *4 }:7cm) {5};
		\node[xshift=0.25cm, yshift=0.08cm] at ({360/6 *5}:7cm) {3};

  \node[vertex] (v7) at ({360/6*0 }:5cm) {};
  \node[vertex] (v12) at ({360/6*1 }:5cm) {};
  \node[vertex] (v11) at ({360/6*2 }:5cm) {};
  \node[vertex] (v10) at ({360/6*3 }:5cm) {};
  \node[vertex] (v9) at ({360/6*4 }:5cm) {};
  \node[vertex] (v8) at ({360/6*5 }:5cm) {};
  
    \draw ({360/6 *0 +1.5}:5cm) ;
    \draw ({360/6 *1 +1.5}:5cm) ;
    \draw ({360/6 *2 +1.5}:5cm) ;
    \draw ({360/6 *3 +1.5}:5cm) ;
    \draw ({360/6 *4 +1.5}:5cm) ;
    \draw ({360/6 *5 +1.5}:5cm) ;
    
   \node[xshift=0.3cm] at ({360/6 *0}:5cm) {7};
   \node[xshift=-0.1cm, yshift = -0.35cm] at ({360/6 *1}:5cm) {12};
   \node[ yshift=-0.3cm] at ({360/6 *2}:5cm) {11};
   \node[xshift=0.35cm, yshift=-0.05cm] at ({360/6 *3}:5cm) {10};
   \node[ yshift=-0.3cm] at ({360/6 *4}:5cm) {9};
   \node[xshift=0.05cm, yshift=-0.35cm] at ({360/6 *5}:5cm) {8};

 \foreach \from/\to in {v4/v10,v3/v9,v1/v7,v6/v12,v7/v9,v8/v10,v9/v11,v10/v12,v11/v7,v12/v8} \draw (\from) -- (\to);
 
 \foreach \from/\to in {v2/v8,v5/v11} \draw [dashed] (\from) -- (\to);

 \end{scope}
 
  \end{tikzpicture}
\caption{A parity signed D\" urer graph with four negative edges. Solid lines denote positive edges and dashed lines denote negative edges.}\label{P_6,2}
\end{center}
\end{figure}
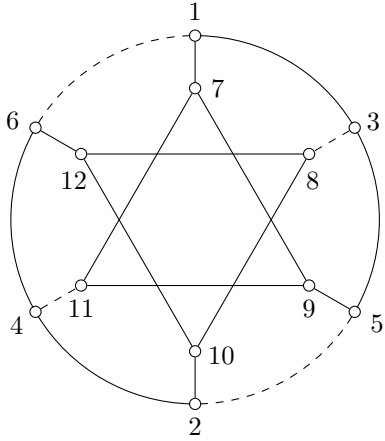

\begin{example}\label{the rna number of P(6,2)} \rm{ The generalized Petersen graph $P(6,2)$ is known as the D\" urer graph. It is depicted in Figure~\ref{P_6,2}.
By Remark~\ref{remark 1}, we have $\sigma^{-}(P(6,2)) \geq 4$. Let $f : V(P(6,2)) \rightarrow \{1,2,...,12\}$ be defined by $ f(v_{i}) = i+7, ~\text{for}~ i = 0,1,2,3,4,5$ and
$$\hspace{-0.19in} f(u_{i}) = 
 \begin{cases}
2i+1 \hspace{0.2in} ~\text{for}~ i=0,1,2,\\
2i- 4 \hspace{0.2in} ~\text{for}~ i=3,4,5.
\end{cases}
$$
This vertex labeling of $P(6,2)$ is described in Figure~\ref{P_6,2}. Clearly $(P(6,2), \sigma_{f})$ have four negative edges as shown in Figure~\ref{P_6,2}. Hence the \textbf{rna} number of D\" urer graph is four.}
\end{example}

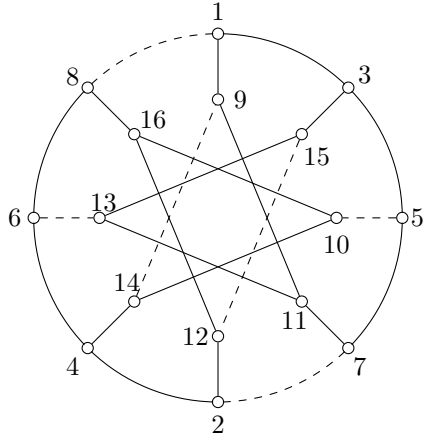
\begin{figure}[h]
\begin{center}
\begin{tikzpicture}[scale=0.35]
 
 \begin{scope}[rotate=90]

		\node[vertex] (v1) at ({360/8*0 }:7cm) {};
		\node[vertex] (v8) at ({360/8*1 }:7cm) {};
		\node[vertex] (v7) at ({360/8 *2 }:7cm) {};
		\node[vertex] (v6) at ({360/8 *3 }:7cm) {};
		\node[vertex] (v5) at ({360/8 *4 }:7cm) {};
		\node[vertex] (v4) at ({360/8 *5}:7cm) {};
		\node[vertex] (v3) at ({360/8 *6 }:7cm) {};
		\node[vertex] (v2) at ({360/8 *7}:7cm) {};
        \draw [dashed] ({360/8 *0 +1.5}:7cm) arc  ({360/8*0+1.5}:{360/8*1 -1.5 }:7cm);
        \draw ({360/8 *1 +1.5}:7cm) arc  ({360/8*1+1.5}:{360/8*2 -1.5 }:7cm);
        \draw ({360/8 *2 +1.5}:7cm) arc  ({360/8*2+1.5}:{360/8*3 -1.5 }:7cm);
        \draw ({360/8 *3 +1.5}:7cm) arc  ({360/8*3+1.5}:{360/8*4 -1.5 }:7cm);
        \draw [dashed] ({360/8 *4 +1.5}:7cm) arc  ({360/8*4+1.5}:{360/8*5 -1.5 }:7cm);
        \draw ({360/8*5 +1.5}:7cm) arc  ({360/8*5+1.5}:{360/8*6 -1.5 }:7cm);
        \draw ({360/8 *6 +1.5}:7cm) arc  ({360/8*6+1.5}:{360/8*7 -1.5 }:7cm);
        \draw ({360/8*7 +1.5}:7cm) arc  ({360/8*7+1.5}:{360/8*8 -1.5 }:7cm);
        
        \node[ yshift=0.3cm] at ({360/8*0 }:7cm) {1};
		\node[xshift=-0.2cm, yshift=0.15cm,] at ({360/8*1 }:7cm) {8};
		\node[xshift=-0.25cm] at ({360/8 *2 }:7cm) {6};
		\node[xshift=-0.2cm,yshift=-0.25cm] at ({360/8 *3 }:7cm) {4};
		\node[ yshift=-0.3cm] at ({360/8 *4 }:7cm) {2};
		\node[xshift=0.15cm, yshift=-0.25cm] at ({360/8 *5}:7cm) {7};
		\node[xshift=0.2cm] at ({360/8 *6 }:7cm) {5};
		\node[xshift=0.22cm, yshift=0.18cm] at ({360/8 *7}:7cm) {3};

  \node[vertex] (v9) at ({360/8*0 }:4.5cm) {};
  \node[vertex] (v16) at ({360/8*1 }:4.5cm) {};
  \node[vertex] (v15) at ({360/8*2 }:4.5cm) {};
  \node[vertex] (v14) at ({360/8*3 }:4.5cm) {};
  \node[vertex] (v13) at ({360/8*4 }:4.5cm) {};
  \node[vertex] (v12) at ({360/8*5 }:4.5cm) {};
  \node[vertex] (v11) at ({360/8*6 }:4.5cm) {};
  \node[vertex] (v10) at ({360/8*7 }:4.5cm) {};
  
    \draw ({360/8 *0 +1.5}:4.5cm) ;
    \draw ({360/8 *1 +1.5}:4.5cm) ;
    \draw ({360/8 *2 +1.5}:4.5cm) ;
    \draw ({360/8 *3 +1.5}:4.5cm) ;
    \draw ({360/8 *4 +1.5}:4.5cm) ;
    \draw ({360/8 *5 +1.5}:4.5cm) ;
    \draw ({360/8 *6 +1.5}:4.5cm) ;
    \draw ({360/8 *7 +1.5}:4.5cm) ;

   \node[xshift=0.3cm] at ({360/8 *0}:4.5cm) {9};
   \node[xshift=0.25cm, yshift=0.2cm] at ({360/8 *1}:4.5cm) {16};
   \node[xshift=0.05cm, yshift=0.2cm] at ({360/8 *2}:4.5cm) {13};
   \node[xshift=-0.1cm, yshift=0.25cm] at ({360/8 *3}:4.5cm) {14};
   \node[xshift=-0.3cm] at ({360/8 *4}:4.5cm) {12};
   \node[xshift=-0.1cm, yshift=-0.25cm] at ({360/8 *5}:4.5cm) {11};
   \node[ yshift=-0.35cm] at ({360/8 *6}:4.5cm) {10};
   \node[xshift=0.2cm, yshift=-0.3cm] at ({360/8 *7}:4.5cm) {15};
  
 \foreach \from/\to in {v4/v12,v2/v10,v1/v9,v8/v16,v6/v14,v5/v13,v15/v10,v16/v11,v13/v16,v11/v14,v12/v15,v9/v12} \draw (\from) -- (\to);
   
 \foreach \from/\to in {v3/v11,v7/v15,v14/v9, v10/v13} \draw [dashed] (\from) -- (\to);
  
\end{scope}  
  
  \end{tikzpicture}
\caption{Parity signed M\" obius-Kantor graph with six negative edges}\label{P_8,3}
\end{center}
\end{figure}

\begin{example}\label{the rna number of P(8,3)} \rm{The generalized Petersen graph $P(8,3)$ is known as the \textit{M\" obius-Kantor} graph. A  parity signed M\" obius-Kantor is depicted in Figure~\ref{P_8,3}. By Theorem~\ref{rna number of P(2l,k) with gcd(n,k) = 1}, we have $\sigma^{-}(P(8,3)) \geq 6$. Let $f : V(P(8,3)) \rightarrow \{1,2,...,16\}$ be defined by $ f(v_{0}) = 9,~f(v_{1}) = 15,~f(v_{2}) = 10,~f(v_{3}) = 11,~f(v_{4}) = 12,~f(v_{5}) = 14,~f(v_{6}) = 13,~f(v_{7}) = 16$ and
$$\hspace{-0.19in} f(u_{i}) = 
 \begin{cases}
2i+1 \hspace{0.2in} ~\text{for}~ i=0,1,2,3,\\
2i- 6 \hspace{0.2in} ~\text{for}~ i=4,5,6,7.
\end{cases}
$$
The vertex labeling $f$ is shown in Figure~\ref{P_8,3}. We see that $(P(8,3), \sigma_{f})$ has six negative edges. Hence the \textbf{rna} number of M\" obius-Kantor graph is six.}
\end{example}

\begin{example}\label{the rna number of Dodecahedron} \rm{The generalized Petersen graph $P(10,2)$ is known as the \textit{Dodecahedron} graph. A parity signed Dodecahedron with six negative edges is depicted in Figure~\ref{P_10,2}. This labeling is the one described in the proof of Theorem~\ref{rna number of P(2l,2)}. Clearly, the \textbf{rna} number of Dodecahedron graph is six.}
\end{example}

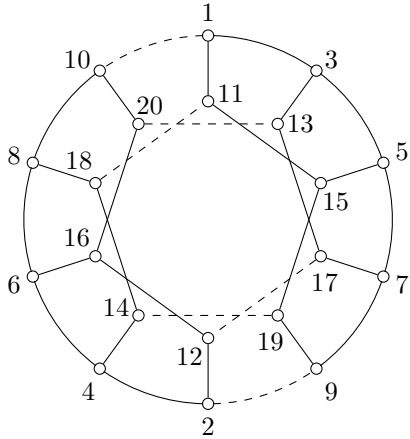
\begin{figure}[h]
\begin{center}
\begin{tikzpicture}[scale=0.35]

   \begin{scope}[rotate=90]
   
		\node[vertex] (v1) at ({360/10*0 }:7cm) {};
		\node[vertex] (v10) at ({360/10*1 }:7cm) {};
		\node[vertex] (v9) at ({360/10 *2 }:7cm) {};
		\node[vertex] (v8) at ({360/10 *3 }:7cm) {};
		\node[vertex] (v7) at ({360/10 *4 }:7cm) {};
		\node[vertex] (v6) at ({360/10 *5}:7cm) {};
		\node[vertex] (v5) at ({360/10 *6 }:7cm) {};
		\node[vertex] (v4) at ({360/10 *7}:7cm) {};
		\node[vertex] (v3) at ({360/10 *8}:7cm) {};
		\node[vertex] (v2) at ({360/10 *9}:7cm) {};
        \draw [dashed] ({360/10 *0 +1.5}:7cm) arc  ({360/10*0+1.5}:{360/10*1 -1.5 }:7cm);
        \draw ({360/10 *1 +1.5}:7cm) arc  ({360/10*1+1.5}:{360/10*2 -1.5 }:7cm);
        \draw  ({360/10 *2 +1.5}:7cm) arc  ({360/10*2+1.5}:{360/10*3 -1.5 }:7cm);
        \draw  ({360/10 *3 +1.5}:7cm) arc  ({360/10*3+1.5}:{360/10*4 -1.5 }:7cm);
        \draw ({360/10 *4 +1.5}:7cm) arc  ({360/10*4+1.5}:{360/10*5 -1.5 }:7cm);
        \draw [dashed] ({360/10*5 +1.5}:7cm) arc  ({360/10*5+1.5}:{360/10*6 -1.5 }:7cm);
        \draw  ({360/10 *6 +1.5}:7cm) arc  ({360/10*6+1.5}:{360/10*7 -1.5 }:7cm);
        \draw  ({360/10*7 +1.5}:7cm) arc  ({360/10*7+1.5}:{360/10*8 -1.5 }:7cm);
        \draw ({360/10 *8 +1.5}:7cm) arc  ({360/10*8+1.5}:{360/10*9 -1.5 }:7cm);
        \draw ({360/10 *9 +1.5}:7cm) arc  ({360/10*9+1.5}:{360/10*10 -1.5 }:7cm);
        
        \node[yshift=0.3cm] at ({360/10*0 }:7cm) {1};
		\node[xshift=-0.3cm,yshift=0.15cm] at ({360/10*1 }:7cm) {10};
		\node[xshift=-0.25cm, yshift=0.1cm] at ({360/10 *2 }:7cm) {8};
		\node[xshift=-0.25cm,yshift=-0.1cm] at ({360/10 *3 }:7cm) {6};
		\node[xshift=-0.15cm, yshift=-0.3cm] at ({360/10 *4 }:7cm) {4};
		\node[ yshift=-0.3cm] at ({360/10 *5}:7cm) {2};
		\node[xshift=0.2cm, yshift=-0.3cm] at ({360/10 *6 }:7cm) {9};
		\node[xshift=0.25cm, yshift=-0.1cm] at ({360/10 *7}:7cm) {7};
		\node[xshift=0.25cm, yshift=0.15cm] at ({360/10 *8}:7cm) {5};
		\node[xshift=0.2cm,  yshift=0.15cm] at ({360/10 *9}:7cm) {3};
  
  \node[vertex] (v11) at ({360/10*0 }:4.5cm) {};
  \node[vertex] (v20) at ({360/10*1 }:4.5cm) {};
  \node[vertex] (v19) at ({360/10*2 }:4.5cm) {};
  \node[vertex] (v18) at ({360/10*3 }:4.5cm) {};
  \node[vertex] (v17) at ({360/10*4 }:4.5cm) {};
  \node[vertex] (v16) at ({360/10*5 }:4.5cm) {};
  \node[vertex] (v15) at ({360/10*6 }:4.5cm) {};
  \node[vertex] (v14) at ({360/10*7 }:4.5cm) {};
  \node[vertex] (v13) at ({360/10*8 }:4.5cm) {};
  \node[vertex] (v12) at ({360/10*9 }:4.5cm) {};
  
    \draw ({360/10 *0 +1.5}:4.5cm) ;
    \draw ({360/10 *1 +1.5}:4.5cm) ;
    \draw ({360/10 *2 +1.5}:4.5cm) ;
    \draw ({360/10 *3 +1.5}:4.5cm) ;
    \draw ({360/10 *4 +1.5}:4.5cm) ;
    \draw ({360/10 *5 +1.5}:4.5cm) ;
    \draw ({360/10 *6 +1.5}:4.5cm) ;
    \draw ({360/10 *7 +1.5}:4.5cm) ;
    \draw ({360/10 *8 +1.5}:4.5cm) ;
    \draw ({360/10 *9 +1.5}:4.5cm) ;

   \node[xshift=0.3cm, yshift=0.1cm] at ({360/10 *0}:4.5cm) {11};
   \node[xshift=0.15cm, yshift=0.25cm] at ({360/10 *1}:4.5cm) {20};
   \node[xshift=-0.22cm, yshift=0.3cm] at ({360/10 *2}:4.5cm) {18};
   \node[xshift=-0.25cm, yshift=0.25cm] at ({360/10 *3}:4.5cm) {16};
   \node[xshift=-0.3cm, yshift=0.1cm] at ({360/10 *4}:4.5cm) {14};
   \node[xshift=-0.25cm, yshift=-0.25cm] at ({360/10 *5}:4.5cm) {12};
   \node[xshift=-0.1cm, yshift=-0.35cm] at ({360/10 *6}:4.5cm) {19};
   \node[xshift=0.05cm, yshift=-0.35cm] at ({360/10 *7}:4.5cm) {17};
   \node[xshift=0.2cm, yshift=-0.25cm] at ({360/10 *8}:4.5cm) {15};
   \node[xshift=0.3cm] at ({360/10 *9}:4.5cm) {13};
  
 \foreach \from/\to in {v4/v14,v3/v13,v2/v12,v1/v11,v10/v20,v9/v19,v8/v18,v7/v17,v6/v16,v5/v15,v13/v15,v16/v18,v17/v19,v18/v20,v11/v13,v12/v14} \draw (\from) -- (\to);
 
\foreach \from/\to in {v19/v11,v20/v12,v14/v16,v15/v17} \draw [dashed] (\from) -- (\to); 
 
 \end{scope}
  
  \end{tikzpicture}
\caption{Parity signed Dodecahedron graph with six negative edges}\label{P_10,2}
\end{center}
\end{figure}

\begin{example}\label{the rna number of Desargues}
\rm{The generalized Petersen graph $P(10,3)$ is known as the \textit{Desargues} graph. See Figure~\ref{P_10,3} for a parity signed Desargues graph. Since $\gcd(10,3)=1$, Theorem~\ref{rna number of P(2l,k) with gcd(n,k) = 1} gives $\sigma^{-}(P(10,3)) \geq 6$. Let the labeling $f : V(P(10,3)) \rightarrow \{1,2,...,20\}$ be defined by $f(v_{0}) = 11,~f(v_{1}) = 13,~f(v_{2}) = 12,~f(v_{3}) = 15,~f(v_{4}) = 17,~f(v_{5}) = 14,~f(v_{6}) = 16,~f(v_{7}) = 19,~f(v_{8}) = 18,~f(v_{9}) = 20$, and 
$$ \hspace{0.1in} f(u_{i}) = 
\begin{cases}
2i+1 \hspace{0.2in} ~\text{for}~ i=0,1,2,3,4,\\
2i-8 \hspace{0.2in} ~\text{for}~ i=5,6,7,8,9.
\end{cases}
$$
This vertex labeling is shown in Figure~\ref{P_10,3}. It is clear that $(P(10,3), \sigma_{f})$ has six negative edges. This gives $\sigma^{-}(P(10,3)) = 6$. Hence the \textbf{rna} number of Desargues graph is six. }
\end{example}

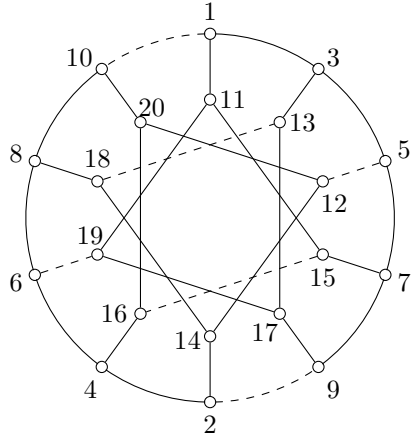
\begin{figure}[h]
\begin{center}
\begin{tikzpicture}[scale=0.35]
  
   \begin{scope}[rotate=90]

		\node[vertex] (v1) at ({360/10*0 }:7cm) {};
		\node[vertex] (v10) at ({360/10*1 }:7cm) {};
		\node[vertex] (v9) at ({360/10 *2 }:7cm) {};
		\node[vertex] (v8) at ({360/10 *3 }:7cm) {};
		\node[vertex] (v7) at ({360/10 *4 }:7cm) {};
		\node[vertex] (v6) at ({360/10 *5}:7cm) {};
		\node[vertex] (v5) at ({360/10 *6 }:7cm) {};
		\node[vertex] (v4) at ({360/10 *7}:7cm) {};
		\node[vertex] (v3) at ({360/10 *8}:7cm) {};
		\node[vertex] (v2) at ({360/10 *9}:7cm) {};
        \draw [dashed] ({360/10 *0 +1.5}:7cm) arc  ({360/10*0+1.5}:{360/10*1 -1.5 }:7cm);
        \draw ({360/10 *1 +1.5}:7cm) arc  ({360/10*1+1.5}:{360/10*2 -1.5 }:7cm);
        \draw ({360/10 *2 +1.5}:7cm) arc  ({360/10*2+1.5}:{360/10*3 -1.5 }:7cm);
        \draw ({360/10 *3 +1.5}:7cm) arc  ({360/10*3+1.5}:{360/10*4 -1.5 }:7cm);
        \draw ({360/10 *4 +1.5}:7cm) arc  ({360/10*4+1.5}:{360/10*5 -1.5 }:7cm);
        \draw [dashed] ({360/10*5 +1.5}:7cm) arc  ({360/10*5+1.5}:{360/10*6 -1.5 }:7cm);
        \draw ({360/10 *6 +1.5}:7cm) arc  ({360/10*6+1.5}:{360/10*7 -1.5 }:7cm);
        \draw ({360/10*7 +1.5}:7cm) arc  ({360/10*7+1.5}:{360/10*8 -1.5 }:7cm);
        \draw ({360/10 *8 +1.5}:7cm) arc  ({360/10*8+1.5}:{360/10*9 -1.5 }:7cm);
        \draw ({360/10 *9 +1.5}:7cm) arc  ({360/10*9+1.5}:{360/10*10 -1.5 }:7cm);
        
       \node[yshift=0.3cm] at ({360/10*0 }:7cm) {1};
		\node[xshift=-0.3cm,yshift=0.15cm] at ({360/10*1 }:7cm) {10};
		\node[xshift=-0.25cm, yshift=0.1cm] at ({360/10 *2 }:7cm) {8};
		\node[xshift=-0.25cm,yshift=-0.1cm] at ({360/10 *3 }:7cm) {6};
		\node[xshift=-0.15cm, yshift=-0.3cm] at ({360/10 *4 }:7cm) {4};
		\node[ yshift=-0.3cm] at ({360/10 *5}:7cm) {2};
		\node[xshift=0.2cm, yshift=-0.3cm] at ({360/10 *6 }:7cm) {9};
		\node[xshift=0.25cm, yshift=-0.1cm] at ({360/10 *7}:7cm) {7};
		\node[xshift=0.25cm, yshift=0.15cm] at ({360/10 *8}:7cm) {5};
		\node[xshift=0.2cm,  yshift=0.15cm] at ({360/10 *9}:7cm) {3};
  
  \node[vertex] (v11) at ({360/10*0 }:4.5cm) {};
  \node[vertex] (v20) at ({360/10*1 }:4.5cm) {};
  \node[vertex] (v19) at ({360/10*2 }:4.5cm) {};
  \node[vertex] (v18) at ({360/10*3 }:4.5cm) {};
  \node[vertex] (v17) at ({360/10*4 }:4.5cm) {};
  \node[vertex] (v16) at ({360/10*5 }:4.5cm) {};
  \node[vertex] (v15) at ({360/10*6 }:4.5cm) {};
  \node[vertex] (v14) at ({360/10*7 }:4.5cm) {};
  \node[vertex] (v13) at ({360/10*8 }:4.5cm) {};
  \node[vertex] (v12) at ({360/10*9 }:4.5cm) {};
  
    \draw ({360/10 *0 +1.5}:4.5cm) ;
    \draw ({360/10 *1 +1.5}:4.5cm) ;
    \draw ({360/10 *2 +1.5}:4.5cm) ;
    \draw ({360/10 *3 +1.5}:4.5cm) ;
    \draw ({360/10 *4 +1.5}:4.5cm) ;
    \draw ({360/10 *5 +1.5}:4.5cm) ;
    \draw ({360/10 *6 +1.5}:4.5cm) ;
    \draw ({360/10 *7 +1.5}:4.5cm) ;
    \draw ({360/10 *8 +1.5}:4.5cm) ;
    \draw ({360/10 *9 +1.5}:4.5cm) ;

   \node[xshift=0.3cm] at ({360/10 *0}:4.5cm) {11};
   \node[xshift=0.15cm, yshift=0.2cm] at ({360/10 *1}:4.5cm) {20};
   \node[ yshift=0.3cm] at ({360/10 *2}:4.5cm) {18};
   \node[xshift=-0.1cm, yshift=0.25cm] at ({360/10 *3}:4.5cm) {19};
   \node[xshift=-0.35cm] at ({360/10 *4}:4.5cm) {16};
   \node[xshift=-0.3cm, yshift=-0.05cm] at ({360/10 *5}:4.5cm) {14};
   \node[xshift=-0.2cm, yshift=-0.25cm] at ({360/10*6}:4.5cm) {17};
   \node[ yshift=-0.3cm] at ({360/10 *7}:4.5cm) {15};
   \node[xshift=0.15cm, yshift=-0.3cm] at ({360/10 *8}:4.5cm) {12};
   \node[xshift=0.3cm] at ({360/10 *9}:4.5cm) {13};
  
 \foreach \from/\to in {v4/v14,v2/v12,v1/v11,v10/v20,v9/v19,v7/v17,v6/v16,v5/v15,v13/v16,v15/v18,v16/v19,v17/v20,v18/v11,v20/v13,v11/v14,v12/v15} \draw (\from) -- (\to);
 
  \foreach \from/\to in {v19/v12,v14/v17,v3/v13,v8/v18} \draw [dashed] (\from) -- (\to);

\end{scope}  
  
  \end{tikzpicture}
\caption{A parity signed Desargues graph with six negative edges}\label{P_10,3}
\end{center}
\end{figure}

The generalized Petersen graph $P(12,5)$ is known as \textit{Nauru} graph. A parity signed Nauru graph is depicted in Figure~\ref{P_12,5}.

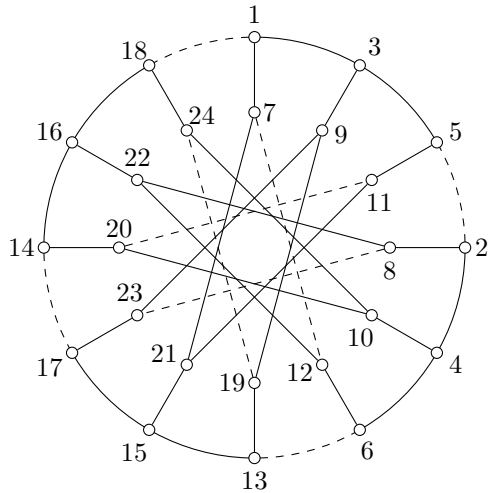
\begin{figure}[h]
\begin{center}
\begin{tikzpicture}[scale=0.40]
 \begin{scope}[rotate=90]
		\node[vertex] (v1) at ({360/12*0 }:7cm) {};
		\node[vertex] (v12) at ({360/12*1 }:7cm) {};
		\node[vertex] (v11) at ({360/12 *2 }:7cm) {};
		\node[vertex] (v10) at ({360/12 *3 }:7cm) {};
		\node[vertex] (v9) at ({360/12 *4 }:7cm) {};
		\node[vertex] (v8) at ({360/12 *5}:7cm) {};
		\node[vertex] (v7) at ({360/12 *6 }:7cm) {};
		\node[vertex] (v6) at ({360/12 *7}:7cm) {};
		\node[vertex] (v5) at ({360/12 *8}:7cm) {};
		\node[vertex] (v4) at ({360/12 *9}:7cm) {};
		\node[vertex] (v3) at ({360/12 *10}:7cm) {};
		\node[vertex] (v2) at ({360/12 *11}:7cm) {};
        \draw [dashed] ({360/12 *0 +1.5}:7cm) arc  ({360/12*0+1.5}:{360/12*1 -1.5 }:7cm);
        \draw ({360/12 *1 +1.5}:7cm) arc  ({360/12*1+1.5}:{360/12*2 -1.5 }:7cm);
        \draw ({360/12 *2 +1.5}:7cm) arc  ({360/12*2+1.5}:{360/12*3 -1.5 }:7cm);
        \draw [dashed] ({360/12 *3 +1.5}:7cm) arc  ({360/12*3+1.5}:{360/12*4 -1.5 }:7cm);
        \draw ({360/12 *4 +1.5}:7cm) arc  ({360/12*4+1.5}:{360/12*5 -1.5 }:7cm);
        \draw ({360/12*5 +1.5}:7cm) arc  ({360/12*5+1.5}:{360/12*6 -1.5 }:7cm);
        \draw [dashed] ({360/12 *6 +1.5}:7cm) arc  ({360/12*6+1.5}:{360/12*7 -1.5 }:7cm);
        \draw ({360/12*7 +1.5}:7cm) arc  ({360/12*7+1.5}:{360/12*8 -1.5 }:7cm);
        \draw  ({360/12 *8 +1.5}:7cm) arc  ({360/12*8+1.5}:{360/12*9 -1.5 }:7cm);
        \draw [dashed] ({360/12 *9 +1.5}:7cm) arc  ({360/12*9+1.5}:{360/12*10 -1.5 }:7cm);
        \draw ({360/12 *10 +1.5}:7cm) arc  ({360/12*10+1.5}:{360/12*11 -1.5 }:7cm);
        \draw  ({360/12*11+1.5}:7cm) arc  ({360/12*11+1.5}:{360/12*12 -1.5 }:7cm);
        
        \node[ yshift=0.3cm] at ({360/12*0 }:7cm) {1};
		\node[xshift=-0.2cm, yshift=0.2cm] at ({360/12*1 }:7cm) {18};
		\node[xshift=-0.3cm, yshift=0.15cm] at ({360/12 *2 }:7cm) {16};
		\node[xshift=-0.3cm] at ({360/12 *3 }:7cm) {14};
		\node[xshift=-0.3cm, yshift=-0.2cm] at ({360/12 *4 }:7cm) {17};
		\node[xshift=-0.2cm, yshift=-0.3cm] at ({360/12 *5}:7cm) {15};
		\node[ yshift=-0.3cm] at ({360/12 *6 }:7cm) {13};
		\node[xshift=0.1cm, yshift=-0.28cm] at ({360/12 *7}:7cm) {6};
		\node[xshift=0.25cm, yshift=-0.15cm] at ({360/12 *8}:7cm) {4};
		\node[xshift=0.22cm] at ({360/12 *9}:7cm) {2};
		\node[xshift=0.25cm, yshift=0.15cm] at ({360/12 *10}:7cm) {5};
		\node[xshift=0.2cm, yshift=0.25cm] at ({360/12 *11}:7cm) {3};

  \node[vertex] (v13) at ({360/12*0 }:4.5cm) {};
  \node[vertex] (v24) at ({360/12*1 }:4.5cm) {};
  \node[vertex] (v23) at ({360/12*2 }:4.5cm) {};
  \node[vertex] (v22) at ({360/12*3 }:4.5cm) {};
  \node[vertex] (v21) at ({360/12*4 }:4.5cm) {};
  \node[vertex] (v20) at ({360/12*5 }:4.5cm) {};
  \node[vertex] (v19) at ({360/12*6 }:4.5cm) {};
  \node[vertex] (v18) at ({360/12*7 }:4.5cm) {};
  \node[vertex] (v17) at ({360/12*8 }:4.5cm) {};
  \node[vertex] (v16) at ({360/12*9 }:4.5cm) {};
  \node[vertex] (v15) at ({360/12*10 }:4.5cm) {};
  \node[vertex] (v14) at ({360/12*11 }:4.5cm) {};
  
    \draw ({360/12 *0 +1.5}:4.5cm) ;
    \draw ({360/12 *1 +1.5}:4.5cm) ;
    \draw ({360/12 *2 +1.5}:4.5cm) ;
    \draw ({360/12 *3 +1.5}:4.5cm) ;
    \draw ({360/12 *4 +1.5}:4.5cm) ;
    \draw ({360/12 *5 +1.5}:4.5cm) ;
    \draw ({360/12 *6 +1.5}:4.5cm) ;
    \draw ({360/12 *7 +1.5}:4.5cm) ;
    \draw ({360/12 *8 +1.5}:4.5cm) ;
    \draw ({360/12 *9 +1.5}:4.5cm) ;
    \draw ({360/12 *10 +1.5}:4.5cm) ;
    \draw ({360/12 *11 +1.5}:4.5cm) ;
    
   \node[xshift=0.2cm] at ({360/12 *0}:4.5cm) {7};
   \node[xshift=0.2cm, yshift=0.2cm] at ({360/12 *1}:4.5cm) {24};
   \node[ yshift=0.3cm] at ({360/12 *2}:4.5cm) {22};
   \node[ yshift=0.25cm] at ({360/12 *3}:4.5cm) {20};
   \node[xshift=-0.1cm, yshift=0.3cm] at ({360/12 *4}:4.5cm) {23};
   \node[xshift=-0.3cm, yshift=0.15cm] at ({360/12 *5}:4.5cm) {21};
   \node[xshift=-0.3cm] at ({360/12 *6}:4.5cm) {19};
   \node[xshift=-0.3cm, yshift=-0.1cm] at ({360/12 *7}:4.5cm) {12};
   \node[xshift=-0.15cm, yshift=-0.3cm] at ({360/12 *8}:4.5cm) {10};
   \node[ yshift=-0.3cm] at ({360/12 *9}:4.5cm) {8};
   \node[xshift=0.1cm, yshift=-0.3cm] at ({360/12 *10}:4.5cm) {11};
   \node[xshift=0.25cm, yshift=-0.05cm] at ({360/12 *11}:4.5cm) {9};

  \foreach \from/\to in {v4/v16,v3/v15,v2/v14,v1/v13,v12/v24,v11/v23,v10/v22,v9/v21,v8/v20,v7/v19,v6/v18,v5/v17,v14/v19,v15/v20,v17/v22,v18/v23,v20/v13,v21/v14,v23/v16,v24/v17} \draw (\from) -- (\to);
  
    \foreach \from/\to in {v13/v18,v22/v15,v16/v21,v19/v24} \draw [dashed] (\from) -- (\to);
  
\end{scope}  
  
  \end{tikzpicture}
\caption{A parity signed Nauru graph with eight negative edges}\label{P_12,5}
\end{center}
\end{figure}

\begin{lemma}\label{lemma for rna number of Nauru}
The \textbf{rna} number of Nauru graph is at least 8.
\end{lemma}
\begin{proof}
Theorem~\ref{rna number of P(2l,k) with gcd(n,k) = 1} gives that $\sigma^{-}(P(12,5))\geq 6$. Also, Lemma~\ref{killing possibility of odd cuts} shows that the \textbf{rna} number of Nauru graph cannot be seven. Thus it remains to prove that $\sigma^{-}(P(12,5))$ cannot be six.

By contradiction, let $P(12,5)$ have a cut of size six with equal sides. Then there exists a subset $A \subset V(P(12,5))$ such that $|A|=12$ and $|[A : A^c]| = 6$. Note that $A$ must contain $u$-vertices as well as $v$-vertices. Clearly, there is only one inner cycle $C_I$ in $P(12,5)$ since $\gcd(12,5) = 1$. Thus the cut $[A : A^c]$ will contain even number of edges of both $C_o$ and $C_I$, because $A$ and $A^c$ contain $u$-vertices as well as $v$-vertices. Hence we analyse two cases.

\textbf{Case 1.} Let $[A : A^c]$ contain four edges of $C_o$ (or $C_I$) and two edges of $C_I$ (or $C_o$). Assume that $[A : A^c]$ contains four edges of $C_o$ and two edges of $C_I$. The case that $[A : A^c]$ contains two edges of $C_0$ and four edges of $C_I$ can be treated similarly. Since $[A : A^c]$ contains no spoke, $A$ has to have six $u$-vertices and six $v$-vertices. 

Note that all $v$-vertices of $A$ induce a path of order six, since exactly two edges of $C_I$ are lying across $A$ and $A^c$. Thus the set of $v$-vertices in $A$ is $\{v_{r},v_{r+5},v_{r+10},v_{r+3},v_{r+8},v_{r+1}\}$, for some $r \in \{0,1,...,11\}$. Consequently, the set of $u$-vertices in $A$ is $\{u_{r},u_{r+5},u_{r+10},u_{r+3},u_{r+8},u_{r+1}\}$, as no spoke of $P(12,5)$ lies in $[A : A^c]$. Now for any $r \in \{0,1,...,11\}$, the edge-cut $[A : A^c]$ contains the following edges of $C_o$: $u_{r+1}u_{r+2}, u_{r+2}u_{r+3}, u_{r+3}u_{r+4}, u_{r+4}u_{r+5}, u_{r+5}u_{r+6}, u_{r+7}u_{r+8}, u_{r+8}u_{r+9}, u_{r+9}u_{r+10}, u_{r+10}u_{r+11}$ and $u_{r+11}u_{r}$. Thus $|[A : A^c]| \geq 12$, a contradiction.

\textbf{Case 2.} Let $[A : A^c]$ contain two edges of $C_o$, two edges of $C_I$ and two spokes. In this case, $A$ (or $A^c$) cannot have more than seven $u$-vertices or $v$-vertices, otherwise the number of spokes in $[A : A^c]$ will exceed two and a contradiction will occur. Hence we consider two sub-cases depending upon whether $A$ has seven $u$-vertices or six $u$-vertices. 

\textit{Subcase 2(i).} Let $A$ have seven $u$-vertices (or seven $v$-vertices) and five $v$-vertices (or five $u$-vertices). Without loss of generality, assume that $A$ has seven $u$-vertices and five $v$-vertices. These seven $u$-vertices and five $v$-vertices induce paths of order seven and five, respectively. Therefore set $A$ must be of form $\{u_{j},u_{j+1},u_{j+2},u_{j+3},u_{j+4},u_{j+5},u_{j+6},v_{r},v_{r+5},v_{r+10},v_{r+3},v_{r+8}\}$, for some $j,r \in \{0,1,...,11\}$. It is easy to check that all five $v$-vertices of $A$ cannot be adjacent to five $u$-vertices of $A$, for any $j~\text{and}~r$. This means at most four $v$-vertices of $A$ can have their partners in $A$. Thus $[A : A^c]$ will contain at least four spokes, a contradiction.

\textit{Subcase 2(ii).} Let $A$ have six $u$-vertices and six $v$-vertices. Since $[A : A^c]$ has only two edges of each $C_o$ and $C_I$, the $u$-vertices and $v$-vertices of $A$ form two paths of order six. Thus $A$ must be $\{u_{j},u_{j+1},u_{j+2},u_{j+3},u_{j+4},u_{j+5},v_{r},v_{r+5},v_{r+10},v_{r+3},v_{r+8},v_{r+1}\}$, for some $j,r \in \{0,1,...,11\}$. For any $j~\text{and}~r$, it is a simple checking that at most four $v$-vertices of $A$ can have their partners in $A$. Therefore $[A : A^c]$ will contain at least four spokes, a contradiction. This completes the proof.
\end{proof}

\begin{example}\label{the rna number of Nauru}
\rm{By Lemma~\ref{lemma for rna number of Nauru}, we have $\sigma^{-}(P(12,5)) \geq 8$. Let $f : V(P(12,5)) \rightarrow \{1,2,...,24\}$ be defined by $$ f(u_{i}) = 
 \begin{cases}
2i+1 \hspace{0.2in} ~\text{for}~ i=0,1,2,6,7,8,\\
2i-4\hspace{0.2in} ~\text{for}~ i=3,4,5,9,10,11;
\end{cases}
$$
and 
$$ \hspace{0.1in} f(v_{i}) = 
\begin{cases}
2i+7 \hspace{0.2in} ~\text{for}~ i=0,1,2,6,7,8,\\
2i+2 \hspace{0.2in} ~\text{for}~ i=3,4,5,9,10,11.
\end{cases}
$$
This vertex labeling is shown in Figure~\ref{P_12,5}. Clearly, $(P(12,5), \sigma_{f})$ has eight negative edges. Hence the \textbf{rna} number of Nauru graph is eight.}
\end{example}

\section{Regular Graphs with \textbf{rna} Number One}\label{section of smallest cubic graphs}

For a given graph $G$, distribution of odd and even integers to the vertices of $G$ is the crucial aspect of determining the \textbf{rna} number of $G$. However, obvious lower and upper bounds for the \textbf{rna} number of $G$ are 1 and $m$, respectively, where $m$ is the size of $G$. It is shown in~\cite[Proposition 4]{Acharya1} that the \textbf{rna} number of a path of order $n \geq 2$ is one. In~\cite{Acharya2}, it is shown that a graph with \textbf{rna} number one must have a cut-edge. More precisely, we have the following theorem. 

\begin{theorem}\cite[Theorem 3.5]{Acharya2}\label{hint 1 for rna number 1}
Let $G$ be a connected graph. Then $\sigma^{-}(G) = 1$ if and only if $G$ has a cut-edge joining two graphs whose orders differ by at most one.
\end{theorem}

Note that having only a cut-edge is not enough for a graph to have \textbf{rna} number one. For example, consider the graph $G$ obtained by adding an edge $e$ to a vertex of the complete graph $K_3$. It is easy to check that $\sigma^{-}(G) = 2$.


An \textit{even regular} graph is a regular graph in which every vertex has even degree. Similarly, an \textit{odd regular} graph is a regular graph in which every vertex has odd degree. Since an even regular graph cannot have a cut-edge, in light of Theorem~\ref{hint 1 for rna number 1}, the \textbf{rna} number of an even graph is at least two. Therefore, the following problem is worth exploring.

\begin{problem}\label{prob for odd regular graph with rna no 1}
For an odd $k \geq 3$, what is the minimum order of a $k$-regular graph whose \textbf{rna} number is one?
\end{problem}

In order to address this problem, first we construct a $(4n-1)$-regular graph with a cut-edge joining two graphs of order $6n-1$ each. 

\begin{lemma}\label{contruction of odd 4n-1 graphs with rna no one}
There exists a $(4n-1)$-regular graph on $12n-2$ vertices with a cut-edge joining two graphs of order $6n-1$ each.
\end{lemma}
\begin{proof}
Let $C_{6n-1}$ be the cycle with $V(C_{6n-1}) = \{v_{0},v_{1},v_{2},...,v_{6n-2}\}$ and $E(C_{6n-1}) = \{v_{i}v_{i+1}~|~ 0 \leq i \leq 6n-2\}$, where subscripts are read modulo $6n-1$. Consider the power graph $C_{6n-1}^{2n-1}$ of $C_{6n-1}$. Note that the degree of each vertex in $C_{6n-1}^{2n-1}$ is $4n-2$. Now for each $1 \leq i \leq 3n-1$, insert an edge between $v_{i}$ and $v_{i+(3n-1)}$ in $C_{6n-1}^{2n-1}$, and denote the graph so obtained by $G_{r}$. Clearly, the order of $G_{r}$ is $6n-1$ and the degree of $v_{0}$ is $4n-2$, whereas the degree of all other vertices in $G_{r}$ is $4n-1$. Now take two disjoint copies of $G_r$ and join their $v_{0}$-vertices by an edge. This resulting graph is the required graph.
\end{proof}

By Theorem~\ref{hint 1 for rna number 1}, the $(4n-1)$-regular graph constructed in Lemma~\ref{contruction of odd 4n-1 graphs with rna no one} has \textbf{rna} number one.

\begin{lemma}\label{lemma 1 of cubic graphs }
A cubic graph of order four cannot have \textbf{rna} number one.
 \end{lemma}
\begin{proof}
The only cubic graph on four vertices is $K_4$, which does not have a cut-edge. Hence the result follows by Theorem~\ref{hint 1 for rna number 1}.
\end{proof}

\begin{lemma}\label{lemma 2 of cubic graphs }
A cubic graph of order six cannot have \textbf{rna} number one.
\end{lemma}
\begin{proof}
Non-isomorphic cubic graphs of order six are those as shown in Figure~\ref{different cubic graphs on 6 vertices}. Clearly, none of these graphs contain a cut-edge.  Thus, in light of Theorem~\ref{hint 1 for rna number 1}, the result follows.
\end{proof}

\begin{figure}[ht]
\begin{subfigure}{0.45\textwidth}
\begin{tikzpicture}[scale=0.55]
\hspace{1.2in}
\node[vertex] (v1) at (19,2) {};
\node[vertex] (v2) at (19,5) {};
\node[vertex] (v3) at (19,8) {};
\node[vertex] (v4) at (24,2) {};
\node[vertex] (v5) at (24,5) {};
\node[vertex] (v6) at (24,8) {};

\foreach \from/\to in {v1/v4,v1/v5,v1/v6,v2/v4,v2/v5,v2/v6,v3/v4,v3/v5,v3/v6} \draw (\from) -- (\to);

\end{tikzpicture}\label{unbal_2}
\end{subfigure}
\begin{subfigure}{0.45\textwidth}
\begin{tikzpicture}[scale=0.55]
\hspace*{0.6in}
\node[vertex] (v1) at (9,2) {};
\node[vertex] (v2) at (14,2) {};
\node[vertex] (v3) at (11.5,4) {};
\node[vertex] (v4) at (11.5,6) {};
\node[vertex] (v5) at (9,8) {};
\node[vertex] (v6) at (14,8) {};

\foreach \from/\to in {v1/v2,v1/v5,v1/v3,v2/v3,v2/v6,v3/v4,v4/v5,v4/v6,v5/v6} \draw (\from) -- (\to);

\end{tikzpicture}
\end{subfigure}
\caption{Non-isomorphic cubic graphs of order six}\label{different cubic graphs on 6 vertices}
\end{figure}
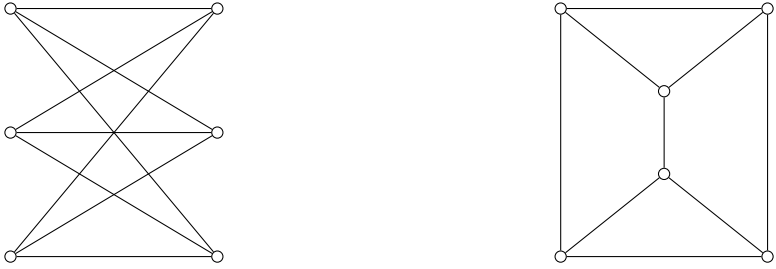

\begin{lemma}\label{lemma 3 of cubic graphs }
A cubic graph of order eight cannot have \textbf{rna} number one.
\end{lemma}
\begin{proof}
There are five non-isomorphic cubic graphs of order eight. These eight graphs are depicted in Figure~\ref{different cubic graphs on 8 vertices}. It is clear that none of these graphs contain a cut-edge. Hence by Theorem~\ref{hint 1 for rna number 1}, the result follows.
\end{proof}

\begin{lemma}\label{cubic graph with rna 1}
There exists a parity signed cubic graph of order ten with \textbf{rna} number one. 
\end{lemma}
\begin{proof}
Let $\Sigma$ be the parity signed cubic graph as shown in Figure~\ref{figure of a Cubic graph of order ten}. Clearly, it is a cubic graph of order 10 and it has a cut-edge joining two graphs of same order. Thus by Theorem~\ref{hint 1 for rna number 1}, we have $\sigma^{-1}(\Sigma) = 1$. This completes the proof.
\end{proof}

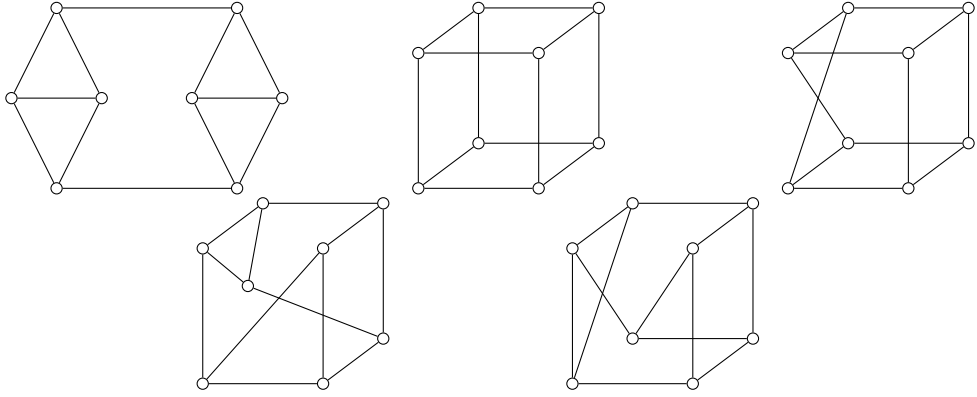
\begin{figure}[ht]
\hspace*{0.3in}
\begin{subfigure}{0.30\textwidth}
\begin{tikzpicture}[scale=0.40]
\node[vertex] (v1) at (9,2) {};
\node[vertex] (v2) at (7.5,5) {};
\node[vertex] (v3) at (10.5,5) {};
\node[vertex] (v4) at (9,8) {};
\node[vertex] (v5) at (15,2) {};
\node[vertex] (v6) at (13.5,5) {};
\node[vertex] (v7) at (16.5,5) {};
\node[vertex] (v8) at (15,8) {};

\foreach \from/\to in {v1/v2,v1/v5,v1/v3,v2/v4,v2/v3,v3/v4,v4/v8,v5/v6,v5/v7,v6/v7,v6/v8,v7/v8} \draw (\from) -- (\to);

\end{tikzpicture}
\end{subfigure}
\hspace*{0.15in}
\begin{subfigure}{0.30\textwidth}
\begin{tikzpicture}[scale=0.40]
\node[vertex] (v1) at (9,2) {};
\node[vertex] (v2) at (9,6.5) {};
\node[vertex] (v3) at (11,3.5) {};
\node[vertex] (v4) at (11,8) {};
\node[vertex] (v5) at (13,2) {};
\node[vertex] (v6) at (13,6.5) {};
\node[vertex] (v7) at (15,3.5) {};
\node[vertex] (v8) at (15,8) {};

\foreach \from/\to in {v1/v2,v1/v5,v1/v3,v2/v4,v2/v6,v3/v4,v4/v8,v5/v6,v5/v7,v3/v7,v6/v8,v7/v8} \draw (\from) -- (\to);

\end{tikzpicture}
\end{subfigure}
\hfill
\begin{subfigure}{0.30\textwidth}
\begin{tikzpicture}[scale=0.40]
\node[vertex] (v1) at (9,2) {};
\node[vertex] (v2) at (9,6.5) {};
\node[vertex] (v3) at (11,3.5) {};
\node[vertex] (v4) at (11,8) {};
\node[vertex] (v5) at (13,2) {};
\node[vertex] (v6) at (13,6.5) {};
\node[vertex] (v7) at (15,3.5) {};
\node[vertex] (v8) at (15,8) {};

\foreach \from/\to in {v1/v4,v1/v5,v1/v3,v2/v4,v2/v6,v3/v2,v4/v8,v5/v6,v5/v7,v3/v7,v6/v8,v7/v8} \draw (\from) -- (\to);

\end{tikzpicture}
\end{subfigure}
\vspace*{0.1in}
\hspace{1.3in}
\begin{subfigure}{0.30\textwidth}
\begin{tikzpicture}[scale=0.40]
\node[vertex] (v1) at (9,2) {};
\node[vertex] (v2) at (9,6.5) {};
\node[vertex] (v3) at (10.5,5.25) {};
\node[vertex] (v4) at (11,8) {};
\node[vertex] (v5) at (13,2) {};
\node[vertex] (v6) at (13,6.5) {};
\node[vertex] (v7) at (15,3.5) {};
\node[vertex] (v8) at (15,8) {};

\foreach \from/\to in {v1/v6,v1/v5,v1/v2,v2/v3,v2/v4,v3/v7,v3/v4,v4/v8,v5/v6,v5/v7,v6/v8,v7/v8} \draw (\from) -- (\to);

\end{tikzpicture} 
\end{subfigure}
\begin{subfigure}{0.30\textwidth}
\begin{tikzpicture}[scale=0.40]
\node[vertex] (v1) at (9,2) {};
\node[vertex] (v2) at (9,6.5) {};
\node[vertex] (v3) at (11,3.5) {};
\node[vertex] (v4) at (11,8) {};
\node[vertex] (v5) at (13,2) {};
\node[vertex] (v6) at (13,6.5) {};
\node[vertex] (v7) at (15,3.5) {};
\node[vertex] (v8) at (15,8) {};

\foreach \from/\to in {v1/v4,v1/v5,v1/v2,v2/v4,v2/v3,v3/v6,v4/v8,v5/v6,v5/v7,v3/v7,v6/v8,v7/v8} \draw (\from) -- (\to);

\end{tikzpicture}
\end{subfigure}
\caption{Non-isomorphic cubic graphs of order eight}\label{different cubic graphs on 8 vertices}
\end{figure}

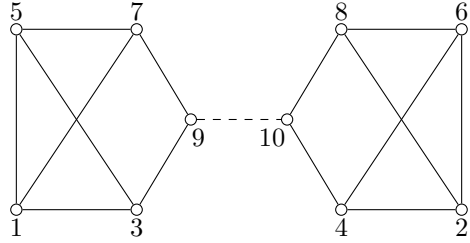
\begin{figure}[ht]
\centering
\begin{tikzpicture}[scale=0.40]
\node[vertex] (v1) at (7,2) {};
\node [below] at (7,2) {1};
\node[vertex] (v2) at (11,2) {};
\node [below] at (11,2) {3};
\node[vertex] (v3) at (7,8) {};
\node [above] at (7,8) {5};
\node[vertex] (v4) at (11,8) {};
\node [above] at (11,8) {7};
\node[vertex] (v5) at (12.8,5) {};
\node [below] [xshift=0.1cm] at (12.8,5) {9};
\node[vertex] (v6) at (16,5) {};
\node [below] [xshift=-0.2cm] at (16,5) {10};
\node[vertex] (v7) at (17.8,2) {};
\node [below] at (17.8,2) {4};
\node[vertex] (v8) at (21.8,2) {};
\node [below] at (21.8,2) {2};
\node[vertex] (v9) at (17.8,8) {};
\node [above] at (17.8,8) {8};
\node[vertex] (v10) at (21.8,8) {};
\node [above] at (21.8,8) {6};

\foreach \from/\to in {v1/v2,v1/v4,v1/v3,v2/v5,v2/v3,v3/v4,v4/v5,v6/v7,v6/v9,v7/v8,v7/v10,v8/v9,v8/v10,v9/v10} \draw (\from) -- (\to);

\foreach \from/\to in {v5/v6} \draw [dashed] (\from) -- (\to);

\end{tikzpicture}
\caption{Parity signed cubic graph of order ten with exactly one negative edge}\label{figure of a Cubic graph of order ten}
\end{figure}

\begin{theorem}\label{upper bound of 4n-1 regular graph}
The smallest order of a $(4n-1)$-regular graph having \textbf{rna} number one is bounded above by $12n-2$. Moreover, this bound is sharp for $n=1$. 
\end{theorem}
\begin{proof}
The proof of the theorem follows from Lemma~\ref{contruction of odd 4n-1 graphs with rna no one} and Lemma~\ref{cubic graph with rna 1}. 
\end{proof} 

Now we construct a $(4n+1)$-regular graph on $8n+6$ vertices with a cut-edge joining two graphs of order $4n+3$ each. 

\begin{lemma}\label{contruction of odd 4n+1 graphs with rna no one}
There exists a $(4n+1)$-regular graph on $8n+6$ vertices with a cut-edge joining two graphs of order $4n+3$ each.
\end{lemma}
\begin{proof}
Let $C_{4n+3}$ be the cycle with $V(C_{4n+3}) = \{v_{0},v_{1},v_{2},...,v_{4n+2}\}$ and $E(C_{4n+3}) = \{v_{i}v_{i+1}~|~ 0 \leq i \leq 4n+2\}$, where subscripts are read modulo $4n+3$. Consider the power graph $C_{4n+3}^{2n}$ of $C_{4n+3}$. Note that the degree of each vertex in $C_{4n+3}^{2n}$ is $4n$. Now for each $1 \leq i \leq 2n+1$, insert an edge between $v_{i}$ and $v_{i+(2n+1)}$ in $C_{4n+3}^{2n}$, and denote the graph so obtained by $G_{s}$. Clearly, the order of $G_{s}$ is $4n+3$ and the degree of $v_{0}$ is $4n$, whereas the degree of all other vertices in $G_{s}$ is $4n+1$. Now take two disjoint copies of $G_s$ and join their $v_{0}$-vertices by an edge. This resulting graph is the required graph.
\end{proof}

Observe that both sides of the graph constructed in Lemma~\ref{contruction of odd 4n+1 graphs with rna no one} are of order $4n+3$. Also, except one vertex, each vertex of these sides have degree $4n+1$. These sides are the smallest such sides of same order whose joining by a cut-edge produce a $(4n+1)$-regular graph. Hence this graph is the $(4n+1)$-regular graph of smallest order having \textbf{rna} number one. Thus the following theorem is immediate from Lemma~\ref{contruction of odd 4n+1 graphs with rna no one}.

\begin{theorem}\label{upper bound of 4n+1 regular graph}
The minimum order of $(4n+1)$-regular graphs having \textbf{rna} number one is $8n+6$. 
\end{theorem}

\section{Time complexity of computing \textbf{rna} number}\label{section of time complexity}
For basic terminologies related to algorithm and its time complexity, we refer the reader to~\cite{Karumanchi}. Recall that the edge-connectivity $\kappa'(G)$ of a graph $G$ is the size of minimum cut. For any graph $G$ of order $n$ and size $m$, the best time bound for edge-connectivity is $O(m+\min \{\kappa'(G) n^{2}, mn+n^{2} log(n)\})$ due to Nagamochi and Ibaraki~\cite{Nagamochi}. Hence the edge-connectivity $\kappa'(G)$ of a graph $G$ can be computed in polynomial time.

For a given graph $G$ of order $n$, we define a family $\mathcal{A}$ of subsets of $V(G)$ as follows
\begin{equation}\label{Equ of family A}
\mathcal{A} = \{A \subset V(G)~~:~~ |A|= \big\lfloor \frac{n}{2} \big\rfloor\}.
\end{equation}

Clearly the cardinality of $\mathcal{A}$ is $C(n, \lfloor \frac{n}{2} \rfloor)$. Let the collection $\mathcal{B}$ be defined by
\begin{equation}\label{Equ of family B}
\mathcal{B} = \{|[A:A^c]| ~~:~~ A \in \mathcal{A}\}.
\end{equation}

If $|\mathcal{B}|$ denotes the cardinality of $\mathcal{B}$ counting the multiplicities of its elements then $|\mathcal{B}|=|\mathcal{A}|$.

\vspace{0.8cm}
\hrule
\vspace{0.2cm}
\noindent
\textbf{Algorithm 1:} The \textbf{rna} number $\sigma^{-}(G)$
\vspace{0.2cm}
\hrule
\vspace{0.2cm}
\noindent
\textbf{Input:} A graph $G$ of order $n$.\\
\textbf{Output:} The \textbf{rna} number $\sigma^{-}(G)$ of $G$\\
  1: obtain the family $\mathcal{A}$ \\
  2: \textbf{for} $A \in \mathcal{A}$ \textbf{do} \\
  3:    \hspace{1cm}    compute $|[A:A^c]|$ \\
  4: \textbf{end for}\\
  5: obtain the family $\mathcal{B}$ by the numbers obtained in step 3\\
  6: find the smallest number in $\mathcal{B}$ and denote it by $\sigma^{-}(G)$\\
  7: return $\sigma^{-}(G)$.
\vspace{0.2cm}
   \hrule
 
 \vspace{0.2cm}
\begin{lemma}
The number returned by Algorithm 1 is the \textbf{rna} number of $G$.
\end{lemma}
\begin{proof}
Clearly each $A \in \mathcal{A}$ generates a cut $[A:A^c]$ whose sides differ by at most one. Thus $\mathcal{B}$ is the collection of sizes of all possible cuts in $G$ whose sides are nearly equal. Consequently, the smallest element of $\mathcal{B}$ will be the \textbf{rna} number of $G$. Thus the lemma follows.
\end{proof}

\vspace{0.2cm}  
\begin{theorem}\label{time complexity}
The running time of Algorithm 1 is $O(2^n+ n^{\lfloor \frac{n}{2} \rfloor}.)$.
\end{theorem}
\begin{proof}
It is well known that, using the inclusion-exclusion principle, all possible subsets of a given size of a set of $n$ elements can be computed in $O(2^n)$ time. Thus for step 1, we spend $O(2^n)$ time. Since $|\mathcal{A}| = C(n, \lfloor \frac{n}{2} \rfloor)$, step 2 to step 4 can be done in $O(n^{\lfloor \frac{n}{2} \rfloor})$ time. Clearly step 5 can be done in constant time.

It is also known that the smallest element in a set of $n$ numbers can be computed in $O(n)$ time. Thus step 6 takes $O(n^{\lfloor \frac{n}{2} \rfloor})$ time, because $|\mathcal{B}|= C(n, \lfloor \frac{n}{2} \rfloor)$. Hence the overall running time of Algorithm 1 is $O(2^n+ n^{\lfloor \frac{n}{2} \rfloor})$. This completes the proof.
\end{proof}

\section{Concluding Remarks}\label{section of conclusion}
We have determined the \textbf{rna} number of $P(n,k)$ for $k=1,2$. For $k \geq 3$, the distribution of odd and even integers to the vertices of $P(n,k)$ to obtain the exact value of $\sigma^{-}(P(n,k))$ seems to be hard. Thus it would be nice if one can solve the following problem.

\begin{problem}\label{prob for the rna no of remaining gen Petersen graphs}
Let $n \geq 7$ and $k \geq 3$ be any given positive integers. What is the value of $\sigma^{-}(P(n,k))$?
\end{problem}

We have proved that the minimum order of a $(4n+1)$-regular graph having \textbf{rna} number one is $8n+3$. We have also proved that the minimum order of a $(4n-1)$-regular graph having \textbf{rna} number one is bounded above by $12n-2$. We could prove the sharpness of this bound only for $n=1$. For $n\geq2$, it is not known if this bound is sharp. From these, we also see that the best possible lower bound for the \textbf{rna} number of odd regular graphs is 1. For each odd $k\geq 3$, best possible upper bound for the \textbf{rna} number of $k$-regular graphs is unknown. Hence the following problems are also interesting.

\begin{problem}\label{prob sharpness}
For $n\geq 2$, what is the minimum order of a $(4n-1)$-regular graph having \textbf{rna} number one?
\end{problem}

\begin{problem}\label{prob for upper bound of the rna no of cubic graphs}
What is the best possible upper bound for the \textbf{rna} number of odd regular graphs?
\end{problem}

We have proposed an exponential time algorithm to find the \textbf{rna} number of a graph $G$. However, there is a minor difference between the concept of the edge-connectivity and the \textbf{rna} number of a graph. So, we propose the following conjecture:

\begin{conjecture}\label{conjecture 1}
The \textbf{rna} number of a graph can be computed in polynomial time. 
\end{conjecture}


\begin{thebibliography}{99}

\bibitem{Acharya1}
M. Acharya and J. V. Kureethara, Parity labeling in signed graphs, J. Prime Research in Math., 17(2) (2021), 1-7. 

\bibitem{Acharya2}
M. Acharya, J. V. Kureethara and T. Zaslavsky, Characterizations of some parity signed graphs, Australasian J. Combinatorics, 81(1) (2021), 89-100.

\bibitem{Bondy}
J. A. Bondy, U. S. R. Murty, \textit{Graph Theory}, Graduate Text in Mathematics, Springer-Verlag, London, 2007.

\bibitem{Coxeter}
H. S. M. Coxeter, Self-dual configurations and regular graphs, Bulletin of the American Mathematical Society, \textbf{56} (5) (1950), 413-455.

\bibitem{Harary} 
F. Harary, On the notion of balance of a signed graph, Michigan Math. J., \textbf{2} (6) (1953), 143-146.

\bibitem{Karumanchi} 
N. Karumanchi, \textit{Data Structures and Algorithms Made Easy}, CareerMonk Publication, India, 2017.

\bibitem{Nagamochi}
H. Nagamochi and T. Ibaraki, Computing edge-connectivity of multigraphs and capacitated graphs, SIAM J. Discrete Math. \textbf{5} (1992), 54-66.

\bibitem{Watkins}
M. E. Watkins, A theorem on Tait Coloring with an application to the generalized Petersen graphs, J. Combinatorial Theory \textbf{6} (2) (1969), 152-164.

\end{thebibliography}
\end{document}